\documentclass{amsart}

\usepackage[english]{babel}
\usepackage[letterpaper,top=2cm,bottom=2cm,left=3cm,right=3cm,marginparwidth=1.75cm]{geometry}
\usepackage{amssymb}
\usepackage{amsmath}
\usepackage{amsthm}
\usepackage{thmtools}
\usepackage{graphicx}
\usepackage{physics}
\usepackage{tikz-cd}
\usepackage{mathtools}
\usepackage[all]{xy}
\usepackage{xcolor}

\theoremstyle{plain}
\newtheorem{theorem}{Theorem}[section]
\newtheorem{corollary}[theorem]{Corollary}
\newtheorem{lemma}[theorem]{Lemma}
\newtheorem{proposition}[theorem]{Proposition}
\theoremstyle{definition}
\newtheorem{definition}[theorem]{Definition}
\newtheorem{remark}[theorem]{Remark}
\newtheorem{example}{Example}[section]

\usepackage[colorlinks=true,allcolors=blue]{hyperref}
\usepackage{cleveref}

\DeclareMathOperator{\Stab}{Stab}
\DeclareMathOperator{\gldim}{gldim}
\DeclareMathOperator{\Hom}{Hom}
\DeclareMathOperator{\Ext}{Ext}
\DeclareMathOperator{\Aut}{Aut}
\DeclareMathOperator{\Coh}{Coh}
\DeclareMathOperator{\Pic}{Pic}
\DeclareMathOperator{\PreStab}{PreStab}
\DeclareMathOperator{\Gd}{Gd}
\DeclareMathOperator{\algldim}{gl.dim}

\DeclareMathOperator{\Geo}{Geo}
\DeclareMathOperator{\PSL}{PSL}
\DeclareMathOperator{\ch}{ch}

\DeclareMathOperator{\modcat}{mod}
\DeclareMathOperator{\rk}{rk}

\newcommand{\D}{\mathcal{D}}
\newcommand{\Db}{\mathrm{D}^{\mathrm b}}
\newcommand{\cP}{\mathcal{P}}
\newcommand{\A}{\mathcal{A}}
\newcommand{\B}{\mathcal{B}}
\newcommand{\cM}{\mathcal{M}}
\newcommand{\cF}{\mathcal{F}}
\newcommand{\cQ}{\mathcal{Q}}
\newcommand{\cT}{\mathcal{T}}
\newcommand{\cV}{\mathcal{V}}
\newcommand{\cW}{\mathcal{W}}
\newcommand{\OO}{\mathcal{O}}
\newcommand{\C}{\mathbb{C}}
\newcommand{\R}{\mathbb{R}}
\newcommand{\Z}{\mathbb{Z}}
\newcommand{\QQ}{\mathbb{Q}}
\newcommand{\HH}{\mathbb{H}}
\newcommand{\PP}{\mathbb{P}}
\newcommand{\FF}{\mathbb{F}}
\newcommand{\ii}{\mathrm{i}}
\newcommand{\bP}{\PP^{1}}
\newcommand{\pre}{\mathrm{pre}}
\newcommand{\Sfun}{\mathbb{S}}
\newcommand{\Sdimup}{\overline{\operatorname{Sdim}}}

\title{The Global Dimension Function on Stability Manifolds}

\author[D. Wu]{Dongjian Wu}
\address{D. Wu: Shanghai Institute for Mathematics and Interdisciplinary Sciences (SIMIS),
Shanghai 200433, China}
\email{wdj@simis.cn}

\author[N. Zhang]{Nantao Zhang}
\address{N. Zhang: School of Mathematics, Sun Yat-sen University, Guangzhou 510275, China}
\email{zhangnt@mail.sysu.edu.cn}

\date{}

\begin{document}

\begin{abstract}
In this paper, we study the reachability and boundary behavior of the global dimension
function on Bridgeland stability manifolds.  For a smooth projective variety
$X$ of dimension $n$, we prove that the infimum of the global dimension
function is $n$.  This value is attained when $-K_X$ is ample, is taken by
every stability condition when $K_X\in\Pic^0(X)$, and is not attained when
$K_X$ is big and nef. We also extend these reachability results to pre-stability conditions. In addition, we construct compactifications of reduced
stability spaces of smooth projective curves to which the global dimension
function extends continuously. Finally, we discuss algebraic models of the reachability problem, including finite-dimensional algebras and semiorthogonal decompositions.
\end{abstract}
\maketitle

\section{Introduction}
Bridgeland stability conditions were introduced as a mathematical formulation of Douglas's notion of \(\Pi\)-stability for D-branes \cite{Bridgeland}.  They extend classical slope stability from abelian categories to triangulated categories.  A stability condition \(\sigma=(\cP,Z)\) consists of a central charge \(Z\) and a slicing \(\cP=\{\cP(\phi)\mid \phi\in\R\}\), where each \(\cP(\phi)\) is an additive subcategory whose nonzero objects are semistable of phase \(\phi\).  A fundamental feature of Bridgeland's theory is that \(\Stab(\D)\) carries a natural structure of a complex manifold, locally modeled on the space of central charges.  Thus \(\Stab(\D)\) is a geometric space which records both wall-crossing phenomena of semistable objects and deformations of central charges.

\subsection{Motivations}
In this paper, we study the global dimension function on a stability manifold, which was
introduced by Ikeda--Qiu in their work on $q$-stability conditions
\cite{IQ2023} and subsequently studied by Qiu in connection with Gepner
equations and the topology of stability spaces
\cite{Qiu,QiuFlow}.  For a stability condition $\sigma=(\cP,Z)$, it is defined by
\[
  \gldim(\sigma)
  =
  \sup\bigl\{
    \phi_2-\phi_1
    \bigm|
    \Hom_{\D}(E_1,E_2)\neq0,
    \ 0\neq E_i\in\cP(\phi_i)
  \bigr\}.
\]
In other words, $\gldim(\sigma)$ measures the largest phase separation across
which a nonzero morphism can occur.  The function is continuous and is
invariant under exact autoequivalences and the natural $\C$-action on
$\Stab(\D)$ \cite{Qiu}.  It therefore descends to the corresponding reduced
quotient of the stability manifold.

When $\Stab(\D)\neq\varnothing$, we write
\[
  \Gd\D
  :=
  \inf_{\sigma\in\Stab(\D)}\gldim(\sigma).
\]
The basic \emph{reachability problem} asks whether this infimum is attained by
a stability condition. Kikuta--Ouchi--Takahashi proved the fundamental lower
bound
\begin{equation}\label{eq:intro-kot-bound}
  \Sdimup\D
  \leq
  \Gd\D,
\end{equation}
where $\Sdimup\D$ is the upper Serre dimension
\cite[Theorem~4.2]{KOT}.  Since
\[
  \Sdimup\bigl(\Db(X)\bigr)=\dim X
\]
for a smooth projective variety $X$, every stability condition in the
relevant stability space has global dimension at least $\dim X$.
Kikuta--Ouchi--Takahashi also showed that, for curves, the behavior of this
lower bound already depends on the canonical class: it is attained for
$\PP^1$, is attained by every stability condition for elliptic curves, and is an unattained infimum for
curves of genus at least two \cite[Theorem~5.16]{KOT}.  This is the starting
point of our geometric reachability problem in arbitrary dimension.

As pointed out in \cite{Bridgeland2006}, the quotient space
\(\Aut(\Db(X))\backslash\Stab(\Db(X))/\C\) is closely related to the
stringy K\"ahler moduli space in mirror symmetry. We therefore expect these
actions to extend to suitable compactifications of stability spaces. In
particular, since the global dimension function is invariant under
\(\Aut(\Db(X))\) and the \(\C\)-action, it is natural to ask whether it
extends continuously to a compactified stability space. The compactification
proposed in \cite{KKO} does not admit such an extension in general. Building
on \cite{HLJR,HLR}, we construct compactifications of stability spaces of
curves that are compatible with \(\gldim\).

\subsection{Main results}
Our first main result determines the infimum of the global dimension function for smooth projective varieties
and describes its reachability in terms of positivity of the canonical bundle. Recall that the value $\Gd\D$ is reachable on $\D$ if there is a
$\sigma\in\Stab(\D)$ such that $\gldim(\sigma)=\Gd\D$.

\begin{theorem}[\Cref{thm:variety}]
\label{thm:intro-varieties}
Let $X$ be a smooth projective variety of dimension $n$ over $\C$. Then
\[
\Gd\bigl(\Db(X)\bigr)=n.
\]
Moreover, the following statements hold.
\begin{enumerate}
    \item If $-K_X$ is ample, then $\Gd\bigl(\Db(X)\bigr)$ is reachable on $\Db(X)$;
    \item If $K_X\in\Pic^0(X)$, then $\gldim$ is constantly equal to $n$ on $\Stab(\Db(X))$;
    \item If $K_X$ is big and nef, then $\Gd\bigl(\Db(X)\bigr)$ is not reachable on $\Db(X)$.
\end{enumerate}
\end{theorem}

Thus the inequality \eqref{eq:intro-kot-bound} is sharp for every smooth
projective variety:
\[
  \Sdimup\bigl(\Db(X)\bigr)
  =
  \Gd\bigl(\Db(X)\bigr)
  =
  \dim X.
\]
The three positivity regimes above have different attainment
behavior.  The condition $K_X\in\Pic^0(X)$ includes the Calabi--Yau case
$K_X\simeq\OO_X$.  The
Fano construction also yields stability conditions on the total space of
$K_X$ supported on the zero section by the inducing strategy of \cite{MaWuZh:2026}.

We also consider the same reachability problem on the pre-stability space. The analogue of \Cref{thm:intro-varieties} remains true for pre-stability conditions with one modification: when \(K_X\in\Pic^0(X)\), we only obtain reachability, but we do not claim constancy for all pre-stability conditions. See \Cref{prop:pre-reachability} for the precise statement. It is clear that $\Stab(\D)\subseteq\PreStab(\D)$ with equality when the support property is automatic. The following shows
that this equality is exceptional and thus passing from stability conditions to pre-stability conditions genuinely enlarges the reachability problem.
\begin{proposition}[{\Cref{prop:support}}]
The following statements hold.
\begin{enumerate}
  \item Let \(Q\) be a finite connected acyclic quiver and
  \(\D_Q=\Db(\modcat Q)\). Then we have
  \[
    \PreStab(\D_Q)=\Stab(\D_Q)
    \quad\Longleftrightarrow\quad
    Q\text{ is Dynkin or Euclidean}.
  \]
  \item Let \(X\) be a connected smooth projective complex variety of
  positive dimension. Then we have
  \[
    \PreStab(\Db(X))=\Stab(\Db(X))
    \quad\Longleftrightarrow\quad X\cong\bP.
  \]
\end{enumerate}
\end{proposition}

Our second main topic is the boundary behavior of $\gldim$ for curves.  For a
smooth projective curve $C$, we consider the reduced quotient
\[
  \cM_C
  :=
  \Aut\bigl(\Db(C)\bigr)\backslash
  \Stab\bigl(\Db(C)\bigr)/\C
\]
and construct compactifications to which $\gldim$ extends continuously. We call a compact Hausdorff compactification of \(\cM_C\) a
\(g\)-compactification if the global dimension function extends continuously
to its boundary (see \Cref{def:g-compactification}).

\begin{theorem}[{\Cref{pro:positive-genus curve} and \Cref{prop:p1-refinement}}]
Let \(C\) be a smooth projective curve.
\begin{enumerate}
\item If \(g(C)=1\), the standard compactified modular curve $X(1)$ is a
  \(g\)-compactification with \(\overline{\gldim}\equiv1\). 
\item If \(g(C)\ge2\), the closed disk \(\overline\cM_C\simeq\overline{\Delta^\ast}\) is a
  \(g\)-compactification where $\overline{\gldim}$ takes value \(1\) at the cusp and \(2\) on the
  outer circle.
 \item If $C=\bP$, \(\overline\cM_{\bP}^{\,g}\) is a \(g\)-compactification of $\cM_{\bP}$. Furthermore, every point on its boundary is represented by a quasi-convergent path.
\end{enumerate}
\end{theorem}

We also include a short algebraic discussion.  For a finite-dimensional
algebra $A$, the standard stability condition realizes the homological global
dimension when $\algldim A<\infty$, while every stability
condition has infinite global dimension when
$\algldim A=\infty$.  For homologically smooth graded gentle
algebras with nonempty stability space, \Cref{ex:surface} combines the surface description of
stability conditions with the computation of the Serre dimension in
\cite{CES} to obtain
\[
  \Sdimup\D=\Gd\D.
\]
We further record that a polarizable nonorthogonal semiorthogonal
decomposition forces the supremum of $\gldim$ to be infinite.  

Together with
\Cref{thm:intro-varieties}, these results verify equality in
\eqref{eq:intro-kot-bound} for smooth projective varieties and for the graded
gentle categories considered here.  As observed in \cite{KOT}, the
inequality is an equality in all standard computed examples. At present, no general equality theorem is known, and we are not aware of any counterexample.

\subsection{Contents}
In \Cref{sec:pre}, we recall Bridgeland stability conditions and the global dimension function.  In \Cref{sec:varieties}, we prove the reachability theorem for stability conditions on smooth projective varieties and then extend this result to pre-stability conditions. In \Cref{sec:compactify}, we study $\gldim$-compatible compactifications of reduced stability spaces of smooth projective curves.  Finally, in \Cref{sec:algebraic}, we discuss the algebraic side of the reachability problem, focusing on finite-dimensional algebras and semiorthogonal decompositions.

\subsection*{AI disclosure}
We acknowledge the use of AI in preparing this manuscript. The authors developed the overall strategy, while ChatGPT aided in some proof verification, language editing, and directed us to \cite{MeinhardtPartsch}, which enabled us to discover \Cref{ex:prestab-without-stab}. The authors take full responsibility for all mathematical claims, arguments, and conclusions, as well as for any remaining errors.

\subsection*{Acknowledgements}
We thank Chunyi Li for posing the initial question that motivated this work and for valuable comments. We are grateful to the Shanghai Center for Mathematical Sciences (SCMS) for its hospitality during the summer school where much of this research was conducted. We also thank Tianle Mao for ongoing discussions, and Zhiyu Liu, Yu Qiu, Xun Lin, Shizhuo Zhang, Yuwei Fan, and Will Donovan for helpful conversations. The first author is particularly indebted to Atsushi Takahashi for his guidance and support during his postdoctoral stay in Japan.

\section{Preliminaries}
\label{sec:pre}
In this section, we recall some basic notions of Bridgeland stability conditions, following \cite{Bridgeland}, and review the global dimension function introduced in \cite{IQ2023} and \cite{Qiu}.

\subsection{Bridgeland stability conditions}
Let $\D$ be a triangulated category and let
$K_0(\D)$ be its Grothendieck group.  Fix a free abelian group $\Lambda$ of
finite rank and a surjective class map
\[
  v:K_0(\D)\twoheadrightarrow\Lambda.
\]

\begin{definition}[{\cite[Definition 3.3]{Bridgeland}}]
A \emph{slicing} $\cP$ on $\D$ is a collection of full additive subcategories
$\cP(\phi)\subset\D$, indexed by $\phi\in\R$, satisfying:
\begin{enumerate}
\item $\cP(\phi+1)=\cP(\phi)[1]$ for every $\phi\in\R$;
\item if $\phi_1>\phi_2$, $E_i\in\cP(\phi_i)$, then
  $\Hom_{\D}(E_1,E_2)=0$;
\item every nonzero object $E\in\D$ admits a finite sequence of distinguished
  triangles 
\[0 =
\xymatrix @C=5mm{
 E_0 \ar[rr]   &&  E_1 \ar[dl] \ar[rr] && E_2 \ar[dl]
 \ar[r] & \dots  \ar[r] & E_{m-1} \ar[rr] && E_m \ar[dl] \\
& A_1 \ar@{-->}[ul] && A_2 \ar@{-->}[ul] &&&& A_m \ar@{-->}[ul]
}
= E
\]
 such that each nonzero factor
    $A_i\in\cP(\phi_i)$ with real numbers $\phi_1>\phi_2>\cdots>\phi_m$.
\end{enumerate}
The sequence in the last condition is called the \emph{Harder--Narasimhan
filtration} (\emph{HN-filtration}) of $E$, and the $A_i$ are its HN factors.
\end{definition}

For $0\neq E\in\D$, we denote the largest and smallest phases
of its HN factors by
\[
  \phi^+_{\cP}(E)=\phi_1,
  \qquad
  \phi^-_{\cP}(E)=\phi_m.
\]
For an interval $I\subset\R$, let $\cP(I)$ be the extension-closed
subcategory generated by $\cP(\phi)$ for $\phi\in I$.  In particular,
$\A_{\cP}:=\cP((0,1])$ is the heart of a bounded $t$-structure. Given two slicings $\cP_1$ and $\cP_2$ on $\D$, their distance is defined as
\[
d(\cP_1,\cP_2)
:=
\sup_{0\neq E\in\D}
\left\{
\left|\phi^+_{\cP_1}(E)-\phi^+_{\cP_2}(E)\right|,
\left|\phi^-_{\cP_1}(E)-\phi^-_{\cP_2}(E)\right|
\right\}.
\]

\begin{definition}[{\cite[Definition 1.1]{Bridgeland}}]
A \emph{Bridgeland pre-stability condition} on $\D$ with respect to
$(\Lambda,v)$ is a pair $\sigma=(\cP,Z)$ consisting of a group homomorphism
$Z:\Lambda\to\C$ and a slicing $\cP$ such that, for every
$0\neq E\in\cP(\phi)$,
\[
  Z(v(E))=m(E)e^{\pi\ii\phi}
  \qquad\text{for some }m(E)>0.
\]
A nonzero object $E$ is called $\sigma$\emph{-semistable} of phase $\phi$ if \(E\in\cP(\phi)\). It is called $\sigma$\emph{-stable} if it is simple in $\cP(\phi)$. A pre-stability condition $\sigma=(\cP,Z)$ is called locally finite if for any $\phi$, there exists an
$\epsilon>0$ such that the quasi-abelian category $\cP(\phi-\epsilon,\phi+\epsilon)$ is finite length. We say a pre-stability condition $\sigma=(\cP,Z)$ satisfies the \emph{support property} if, for any
norm $\lVert\cdot\rVert$ on $\Lambda_{\R}$, there exists a constant $C>0$ such that
\[
  \lVert v(E)\rVert\leq C\lvert Z(v(E))\rvert
\]
for every semistable object $E$. A pre-stability condition satisfying
the support property is called a \emph{stability condition}. 
\end{definition}
Denote the set of all stability conditions with respect to $(\Lambda,v)$ by $\Stab_\Lambda(\D)$. When the support lattice is clear from the context, we simply write $\Stab(\D)$. We also abbreviate $Z(v(E))$ to $Z(E)$ and write
$\phi^\pm_\sigma(E):=\phi^\pm_{\cP}(E)$ and
$\A_\sigma:=\cP((0,1])$.  Equivalently, a stability condition can be described by a heart of a bounded $t$-structure with a stability function satisfying
the Harder--Narasimhan property and the support property
\cite[Proposition~5.3]{Bridgeland}.

For a smooth projective variety $X$, we write $K_{\rm num}(X)$ for the
numerical Grothendieck group of $\Db(X)$. Stability conditions on $\Db(X)$ will always be considered with respect to a
finite-rank support lattice $\Lambda$ of numerical classes, together with its class map
$v$.  Usually \(\Lambda\) is either the full numerical Grothendieck group \(K_{\rm num}(X)\) or a finite-rank lattice naturally determined by the construction under consideration. We write \(\Stab_\Lambda(X)\) when the lattice should be specified, and simply write \(\Stab(X)\) when it is clear from the context.

The set $\Stab(\D)$ carries a natural topology induced by the following generalized metric. For stability conditions $\sigma_1=(\cP_1,Z_1)$ and $\sigma_2=(\cP_2,Z_2)$, their distance is defined by
\[
d(\sigma_1,\sigma_2)
:=
\sup_{0\ne E\in\D}
\left\{
  |\phi^-_{\sigma_2}(E)-\phi^-_{\sigma_1}(E)|,
  |\phi^+_{\sigma_2}(E)-\phi^+_{\sigma_1}(E)|,
  \left|\log\frac{m_{\sigma_2}(E)}{m_{\sigma_1}(E)}\right|
\right\}.
\]
\begin{theorem}[{\cite[Theorem 1.2]{Bridgeland}}]
\label{thm:deformation}
The space of stability conditions $\Stab(\D)$ has the structure of a complex manifold, and the map
\[
\Stab(\D)\to\Hom_{\Z}(\Lambda,\C)
\]
that sends a stability condition to its central charge is a local isomorphism.
\end{theorem}

Finally, we recall two standard group actions on $\Stab(\D)$. Let $\Phi$ be an exact autoequivalence of $\D$, and denote by $\Phi_{\ast}\colon K_0(\D)\to K_0(\D)$ the induced isomorphism on the Grothendieck group. For a stability condition $\sigma=(\cP,Z)$ on $\D$, we define the left action of $\Phi$ on $\sigma$ by 
\[
\Phi\cdot\sigma\coloneq(\Phi\circ\cP,Z\circ\Phi_{\ast}^{-1}).
\]
Let $\widetilde{\mathrm{GL}}_2^+(\R)$ be the universal cover of $\mathrm{GL}_2^+(\R)$. For any element $g=(T,f)\in\widetilde{\mathrm{GL}}_2^+(\R)$, the right $\widetilde{\mathrm{GL}}_2^+(\R)$-action is defined by
\[
\sigma[g]=(\cP[g],Z[g]), \quad Z[g]=T^{-1}Z,\quad \cP[g](\phi)=\cP(f(\phi)).
\]
In particular, for $z=a+\ii b\in\C$, the right $\C$-action on $\sigma$ is defined by
\[
  (\cP,Z)\cdot z
  =
  \bigl(\cP_z,e^{-\pi\ii z}Z\bigr),
  \qquad
  \cP_z(\phi)=\cP(\phi+a).
\]

\subsection{The global dimension function}
\begin{definition}
Let $\cP$ be a slicing on a triangulated category $\D$. Define the \emph{global dimension} of $\cP$ by 
\[
  \gldim(\cP)
  :=
  \sup\bigl\{
    \phi_2-\phi_1
    \bigm|
    \Hom_{\D}\bigl(\cP(\phi_1),\cP(\phi_2)\bigr)\neq0
  \bigr\}
  \in[0,+\infty].
\]
For a stability condition $\sigma=(\cP,Z)$ on $\D$, set
$\gldim(\sigma):=\gldim(\cP)$.
\end{definition}
Thus \(\gldim(\sigma)\) measures the largest phase gap across which a nonzero morphism between semistable objects can occur.
The function $\gldim:\Stab_\Lambda(\D)\to[0,+\infty]$ is continuous and is
invariant under exact autoequivalences and the $\C$-action
\cite[Section~3.1]{Qiu}. Hence it descends to a continuous function on the reduced quotient:
\[
\gldim\colon \Aut(\D)\backslash\Stab(\D)/\C\to[0,\infty].
\]
Assume $\Stab_\Lambda(\D)\neq\varnothing$. We define
\[
  \Gd\D
  :=
  \inf_{\sigma\in\Stab(\D)}\gldim(\sigma),
  \qquad
  \overline{\Gd}\D
  :=
  \sup_{\sigma\in\Stab(\D)}\gldim(\sigma).
\]
The value $\Gd\D$ is called \emph{reachable} if there exists a stability condition $\sigma\in\Stab(\D)$ such that $\gldim(\sigma)=\Gd\D$. Similarly, the supremum \(\overline{\Gd}\D\) is reachable if it is attained by some stability condition.

Let \(\D\) be a smooth proper triangulated category with Serre functor \(\Sfun\), and let \(G\) be a split-generator. For objects \(E,F\in\D\), set
\[
  e_-(E,F)
  :=
  \inf\bigl\{k\in\Z\bigm|\Hom_{\D}(E,F[k])\neq0\bigr\}.
\]
Following \cite{ElaginLunts,KOT}, the \emph{upper Serre dimension} is defined by
\[
  \Sdimup(\D)
  :=
  \limsup_{m\to+\infty}
  \frac{-e_-(G,\Sfun^mG)}{m},
\]
which is independent of the choice of $G$.

\begin{theorem}[{\cite[Theorem~4.2]{KOT}}]
\label{thm:kot-lower-bound}
Let $\D$ be a triangulated category equivalent to a perfect derived
category of a smooth proper dg $\C$-algebra.
If $\Stab(\D)\neq\varnothing$, then
\[
  \Sdimup\D
  \leq
  \Gd\D.
\]
\end{theorem}

If $X$ is smooth projective of dimension $n$, then
\[
  \Sfun_X(E)=E\otimes K_X[n],
  \qquad
  \Sdimup\bigl(\Db(X)\bigr)=n
\]
\cite[Example~2.5]{KOT}. The remaining problem is to determine when this inequality is in fact an equality.

\section{Reachability on smooth projective varieties}
\label{sec:varieties}
In this section, we study the reachability problem for smooth projective
varieties. We first recall phase order and pullback and
pushforward constructions for stability conditions \cite{Li2025,Li2026}. We
then prove the Bayer property for ample classes and use it to establish the reachability theorem for smooth projective varieties. Finally, we explain how the same argument extends from stability conditions to pre-stability conditions.

\subsection{Phase order and induced stability conditions}
In this section, we recall some notions from \cite{Li2025,Li2026}. Let \(\sigma\) and
\(\tau\) be two pre-stability conditions on a triangulated category $\D$. Define relations on pre-stability conditions as follows:
\[
\begin{aligned}
  \sigma\lesssim\tau
  &\Longleftrightarrow
  \cP_\sigma(\phi)\subset\cP_\tau((-\infty,\phi))
  \quad\text{for every }\phi\in\R,
  \label{eq:strict-phase-order}\\
  \sigma\lessapprox\tau
  &\Longleftrightarrow
  \cP_\sigma(\phi)\subset\cP_\tau((-\infty,\phi])
  \quad\text{for every }\phi\in\R.
\end{aligned}
\]
Note that these relations are transitive, i.e. $\sigma_1\lessapprox\sigma_2$ and $\sigma_2\lessapprox\sigma_3$ imply $\sigma_1\lessapprox\sigma_3$.

\begin{lemma}[{\cite[Lemma 4.11]{Li2025}}]
\label{lem:phase-order-tests}
Let $\sigma,\tau\in\Stab(\D)$ and $\Phi\in\Aut(\D)$. Then the following conditions are equivalent:
\begin{enumerate}
  \item \(\sigma\lessapprox\tau\);
  \item for every \(\sigma\)-stable object \(E\), $\phi^+_\tau(E)\le\phi_\sigma(E)$;
  \item for every \(\tau\)-stable object \(F\),
  $
    \phi_\tau(F)\le\phi^-_\sigma(F)$;
    \item $\Phi(\sigma)\lessapprox\Phi(\tau)$.
\end{enumerate}
All statements above hold for $\lesssim$ by replacing $\le$ with $<$.
\end{lemma}

\begin{definition}[{\cite[Definition 3.1]{Li2026}}]
Let \(f:Y\to X\) be a finite morphism between smooth projective varieties. Let \(\sigma_Y=(\cP,Z)\) be a pre-stability condition on $\Db(Y)$. The \emph{pushforward} of $\sigma_Y$ is defined as
\[
  f_\sharp\sigma_Y=(f_\sharp\cP,f_\sharp Z),
\]
where
\[
\begin{aligned}
  f_\sharp\cP(\phi)
  &:=
  \{E\in\Db(X)\mid f^*E\in\cP(\phi)\},
  \label{eq:def-f-sharp-slicing}\\
  f_\sharp Z&:=
  Z\circ f^*\colon K_{\rm num}(X)\xrightarrow[]{f^{\ast}}K_{\rm num}(Y)\to\C.
\end{aligned}
\]
On the other hand, given a pre-stability condition $\sigma_X=(\cP,Z)$ on $\Db(X)$, we define its \emph{pullback} to $\Db(Y)$ by
\[
f^{\sharp}\sigma_X\coloneq (f^{\sharp}\cP,f^\sharp Z),
\]
where
\begin{align*}
f^\sharp\cP(\theta)&\coloneq\{E\in \Db(Y)\mid f_{\ast}E\in\cP(\theta)\},\\
f^{\sharp}Z&\coloneq Z\circ f_{\ast}\colon K_{\rm num}(Y)\xrightarrow[]{f_{\ast}}K_{\rm num}(X)\to\C.
\end{align*}
\end{definition}

\subsection{Bayer property for ample classes}
Let
\[
  E^n=E_1\times\cdots\times E_n,
\]
be the product of $n$ smooth elliptic curves with $E_i\cong E$. For each $1\le i\le n$, fix a point $q_i\in E_i$ and define
\[
H_i\coloneq E_1\times\cdots\times E_{i-1}\times\{q_i\}\times E_{i+1}\times\cdots \times E_n\subset E^n.
\]
Set
\[
H\coloneq H_1+\cdots+H_n.
\]
\begin{theorem}[{\cite[Theorem 5.9]{Liu+2021+135+157},\cite[Theorem 4.5]{li2025stabilityconditionsproductscurves}}]
\label{thm:stab on En}
For every
\((a,b)\in\QQ_{>0}\times\QQ\), there is a unique stability
condition
\[
  \sigma^{a,b}=(\cP^{a,b},Z^{a,b})\in\Stab(E^n)
\]
with
\[
  Z^{a,b}(F)
  =
  -\int_{E^n}e^{-(b+\ii a)H}\ch(F),
  \label{eq:elliptic-central-charge}
\]
and $\phi_{\sigma^{a,b}}(\OO_p)=1$ for any $p\in E^n$.
\end{theorem}

\begin{lemma}
\label{lem:1}
    Let \(E^n\) be the product of elliptic curves. For every \(c \in \QQ_{>0}\) and $(a,b)\in\QQ_{>0}\times\QQ$, the stability condition \(\sigma^{a, b}\) satisfies
    \[\sigma^{a, b} \lessapprox \sigma^{a, b + c}.\]
\end{lemma}

\begin{proof}
    Write \(c = p/q\) with $p,q\in\Z_{>0}$ and set
    \[
    A\coloneq q^2a,\quad B\coloneq q^2b,\quad N\coloneq q^2c=pq.
    \]
    By \cite[Lemma 2.3]{Li2026}, we have $\sigma^{A,B}\lessapprox\sigma^{A,B+N}$.
    Let $\pi_q\colon E^n\to E^n$ denote the multiplication-by-$q$ isogeny. By equation (5.4) in \cite{Li2026}, the pullback along $\pi_q$ acts on the parameters as $\pi^{\sharp}_q\sigma^{a,b}=\sigma^{A,B}$,
    and $\pi^{\sharp}_q\sigma^{a,b+c}=\sigma^{A,B+N}$. Therefore, $\pi^{\sharp}_q\sigma^{a,b}\lessapprox\pi^{\sharp}_q\sigma^{a,b+c}$.
    By \cite[Lemma 3.3 and Lemma 5.3]{Li2026}, we obtain 
    \[
(\pi_q)_{\sharp}\pi^{\sharp}_q\sigma^{a,b}\lessapprox(\pi_q)_\sharp\pi^{\sharp}_q\sigma^{a,b+c}.
    \]
    Finally, by equation (5.5) in \cite{Li2026}, the composition $(\pi_q)_\sharp\pi^\sharp_q$ acts as the identity on the slicing and rescales the central charge by a positive real factor, which does not affect the $\lessapprox$ relation. Thus,  $\sigma^{a,b}\lessapprox\sigma^{a,b+c}$.
\end{proof}
  
\begin{lemma}
\label{lem:2}
    For every $(a,b)\in\QQ_{>0}\times\QQ$ and $c\in\QQ_{>0}$, there are stability conditions \(\sigma^{a,b},\sigma^{a,b+c}\in\Stab(\PP^n)\) such that
    \begin{equation}
    \label{eq:less for En}
      \sigma^{a, b} \lessapprox \sigma^{a, b + c}.  
    \end{equation}
\end{lemma}

\begin{proof}
Take $\sigma^{a,b}\coloneq \pi_\sharp\sigma^{a,b}_{E^n}$, where 
\[
\pi\colon E^n\to E^n/((\Z/2\Z)^n\rtimes\mathfrak S_n)\simeq\PP^n
\]
is the quotient map and $\sigma^{a,b}_{E^n}$ are the stability conditions from \Cref{thm:stab on En}. Since $\pi_\sharp$ preserves \(\lessapprox\) and sends stability conditions to stability conditions \cite[Theorem 6.3]{Li2026}, the result follows from \Cref{lem:1}.
\end{proof}

\begin{lemma}
    \label{lem:global-generated}
    Let \(X\) be a smooth projective variety with an ample line bundle \(H\), and let \(m\) be an integer such that \(mH\) is very ample. Let \(f\colon X\to\PP^n\) be the embedding induced by \(|mH|\), and let \(\sigma=f^{\sharp}\sigma^{a,b}\) be a stability condition from the construction of \cite{Li2026}. For any globally generated line bundle \(L\) on \(X\), we have
    \[\sigma \lessapprox \sigma \otimes L.\]
\end{lemma}

\begin{proof}
    The argument is analogous to that of \cite[Theorem 6.5]{Li2026}.  By \cite[Corollary 6.4]{Li2026}, there exists $a_0$ such that for every $a>a_0$,
\begin{equation}
    \rho^{a,b} \lessapprox \rho^{a, b'} \lesssim \rho^{a,b}[1]
    \label{eq:cmp}
\end{equation}
for \(b < b' \in \QQ\). Then the projection formula gives
\begin{equation}
    \rho^{a,b+1}=\rho^{a,b}\otimes\OO(mH).
    \label{eq:tensor}
\end{equation}. Now fix a $\sigma$-stable object $E\in\cP_{a,b}(\theta)$. Since \(L\) is globally generated, there exists a resolution
\[0 \to F \to \OO(-k_{n - 1} m H)^{\oplus a_{n-1}} \to \cdots \to \OO(-k_1 m H)^{\oplus a_1} \to \OO^{\oplus a_0} \to \OO(L) \to 0\]
with \(k_i,a_i \in \Z\), $a_i>0$, and \(F\) coherent. By tensoring with \(E\), we obtain
\[\phi^{-}_{a,b}(E \otimes \OO(L)) \geq \min\left\{\phi^{-}_{a,b}(E), \min_{1\leq i\leq n-1}\phi^{-}_{a,b}(E \otimes \OO(-k_i m H)[i]), \phi^{-}_{a,b}(E \otimes F[n])\right\}.\]
Using \eqref{eq:cmp} and \eqref{eq:tensor}, we have
\[\phi^{-}_{a,b}(E \otimes \OO(-k_i m H)[i]) \geq \phi_{a,b}(E).\]
Moreover, since all skyscraper sheaves are in \(\cP_{a,b}(1)\), \cite[Lemma 10.1]{Bridgeland2008} implies
\[E \in \Coh(X)[0,n - 1], \quad \Coh(X) \subset \cP_{a,b}((1-n, 1]).\]
Consequently, 
\[E \otimes F[n] \in \Coh(X)[n, 3n - 1] \subset \cP_{a,b}((1, 3n]).\]
In particular, \[\phi_{a, b}^{-}(E \otimes F[n]) > 1 \geq \theta.\]
Thus $\phi^-_{a,b}(E\otimes\OO(L))\geq\theta$. The result follows from \Cref{lem:phase-order-tests}.
\end{proof}

\begin{lemma}
\label{lem:3}
    Let \(X\) be a smooth projective variety with an ample line bundle \(H\), and let \(m\) be an integer such that \(m'H\) is very ample for every \(m'\geq m\). Let \(f\colon X\to\PP^n\) be the embedding induced by \(|mH|\). Then there are stability conditions $\sigma^{a,b}\in\Stab(\PP^n)$ such that
    \begin{equation}
    \label{eq:less for X}
    f^{\sharp}\sigma^{a,b} \otimes \OO(-H) = f^{\sharp}\sigma^{a,b-1/m}.
    \end{equation}
\end{lemma}

\begin{proof}
Let $\sigma^{a,b}\in\Stab(\PP^n)$ for which \Cref{lem:2} applies and set \(\rho^{a,b}=f^{\sharp}\sigma^{a,b}\).
Write \(\phi^{\pm}_{a,b}\coloneq\phi^{\pm}_{\rho^{a,b}}\). Fix \(E\in\cP_{a,b}(\theta)\) with \(\theta\in(0,1]\). By \Cref{lem:global-generated}, we have \(\phi^{-}_{a,b}(E\otimes\OO((m+1)H))\geq \theta\).

On the other hand, since \(\Hom(E\otimes\OO((m+1)H),E\otimes\OO(2mH))\neq0\), the inequalities
\[\phi_{a,b}^-(E \otimes \OO((m + 1)H)) \leq \phi_{a,b}^+(E \otimes \OO(2m H))  < \phi_{a, b}(E) + 1\] and
\[\phi_{a,b}(E)\leq \phi_{a, b - 1 - 1/m}^-(E) < \phi_{a,b}(E) + 1\]
hold. Therefore, 
\[\abs{\phi_{a,b-1-1/m}^-(E)-\phi_{\rho^{a,b}\otimes\OO(-(m+1)H)}^{-}(E)}<1.\]

By \Cref{lem:phase-order-tests}, we have \[\rho^{a,b}\otimes\OO(-(m+1)H)\lessapprox\rho^{a,b}.\] Applying the same argument multiple times, we have
\[\rho^{a,b}[-1]\lesssim\rho^{a,b}\otimes\OO(-m(m+1)H)\lessapprox\rho^{a,b}\otimes\OO(-(m+1)H)\lessapprox\rho^{a,b}.\]
Thus, 
\[\phi_{a,b}(E) \leq \phi_{a,b}^+(E \otimes \OO((m+1)H)) < \phi_{a,b}(E) + 1.\]
Therefore,
\[\abs{\phi_{a,b-1-1/m}^+(E)-\phi_{\rho^{a,b}\otimes\OO(-(m+1)H)}^{+}(E)}<1.\]

Since the two stability conditions \(\rho^{a,b-1-1/m}\) and \(\rho^{a,b}\otimes\OO(-(m+1)H)\) have the same central charge, \cite[Lemma 2.3]{Bridgeland2008} yields
\[\rho^{a,b-1-1/m}=\rho^{a,b}\otimes\OO(-(m+1)H).\]
This implies that
\[\rho^{a,b-1/m}=\rho^{a,b}\otimes\OO(-H),\]
which completes the proof.
\end{proof}

\begin{theorem}
\label{thm:main}
    Let \(X\) be a smooth projective variety and let \(H\) be an ample line bundle. Then there exists a stability condition \(\sigma\) such that
    \[\sigma \otimes \OO(-H) \lessapprox \sigma.\]
\end{theorem}

\begin{proof}
Choose an integer \(m>0\) such that \(mH\) is very ample, and denote by $f\colon X\hookrightarrow\PP^n$ the embedding induced by $|mH|$. By \Cref{lem:2} and \Cref{lem:3}, we obtain a stability condition $\sigma=f^{\sharp}\sigma^{a, b}$ satisfying \eqref{eq:less for En} and \eqref{eq:less for X}. Applying \cite[Lemma 3.1]{Li2026} then gives
\[
\sigma\otimes\OO(-H)=f^{\sharp}\sigma^{a,b}\otimes\OO(-H)=f^{\sharp}\sigma^{a,b-1/m}\lessapprox f^{\sharp}\sigma^{a,b}=\sigma.
\]
\end{proof}

\subsection{Reachability theorem}
In this section, we provide a proof of the following reachability result for smooth projective varieties.
\begin{theorem}
\label{thm:variety}
Let $X$ be a smooth projective variety of dimension $n$ over $\C$. Then we have
\[
\Gd\bigl(\Db(X)\bigr)=n.
\]
Moreover, 
\begin{enumerate}
    \item if $-K_X$ is ample, $\Gd\bigl(\Db(X)\bigr)$ is reachable on $\Db(X)$;
    \item if $K_X\in\Pic^0(X)$, $\gldim$ is constantly equal to $n$ on $\Stab(\Db(X))$;
    \item if $K_X$ is big and nef, $\Gd\bigl(\Db(X)\bigr)$ is not reachable on $\Db(X)$.
\end{enumerate}
\end{theorem}

\begin{lemma}[{\cite[Lemma 4.1]{KOT}}]
\label{lem:kot}
For every nonzero object $E\in\D$, we have
\[
\phi^+_{\sigma}(\Sfun E)-\phi^+_{\sigma}(E)\le\gldim(\sigma)
\]
for all \(\sigma\in\Stab(\D)\).
\end{lemma}

\begin{proposition}
\label{pro:Gd=n}
Let \(X\) be a smooth projective variety of dimension \(n\). Then \(\Gd\Db(X)=n\).
\begin{proof}
By \Cref{thm:kot-lower-bound}, we only need to show that there exists a sequence of stability conditions \(\{\sigma_j\}\subset\Stab(\Db(X))\) such that $\gldim(\sigma_j)\to n$. Let \(\sigma=(\cP,Z)\) be a stability condition on \(\Db(X)\), and set
\[
  \delta_{K_X}(\sigma)
  :=
  d\bigl(\cP,\cP\otimes K_X\bigr),
\]
where \(d\) denotes the standard distance between slicings.  We claim that
\[
  \gldim(\sigma)
  \leq n+\delta_{K_X}(\sigma).
\]
Indeed, let $E\in\cP(\phi)$ and $F\in\cP(\psi)$ with $\Hom(E,F)\neq0$. Then Serre duality yields
\[
  \Hom\bigl(F,E\otimes K_X[n]\bigr)\neq0.
\]
Thus,
\[
  \psi
  \leq
  \phi_{\cP}^{+}\bigl(E\otimes K_X[n]\bigr)
  =
  n+\phi_{\cP}^{+}(E\otimes K_X).
\]
On the other hand, \(E\otimes K_X\) is
\((\cP\otimes K_X)\)-semistable of phase \(\phi\).  Hence
\[
  \phi_{\cP}^{+}(E\otimes K_X)
  \leq
  \phi+\delta_{K_X}(\sigma).
\]
Therefore
\[
  \psi-\phi
  \leq n+\delta_{K_X}(\sigma).
\]
Taking the supremum over all such pairs \((E,F)\) proves the claim.

 We now choose a very ample line bundle \(H\) such that \(H + K_X\) and \(H - K_X\) are both very ample. Consider the embedding \(f\colon X \to {\PP}^N\) induced by \(|H|\) for some $N\in\Z_{>0}$. According to the proof of \Cref{thm:main}, there exists a sequence of stability conditions $\sigma_j\coloneq\sigma^{j^2,0}$ with $\cP_j=\cP^{j^2,0}$ satisfying 
 \[
  \sigma_j \otimes \OO(-H) \lessapprox \sigma_j \lessapprox \sigma_j \otimes \OO(H).
 \]
 Since \(H+K_X\) and \(H-K_X\) are very ample, \Cref{lem:global-generated} also implies
\[
  \sigma_j\otimes\OO_X(-H)
  \lessapprox
  \sigma_j\otimes K_X
  \lessapprox
  \sigma_j\otimes\OO_X(H).
\]
Therefore,
    \[\delta_{K_X}(\sigma_j) \leq d(\cP_j \otimes \OO(-H), \cP_j \otimes \OO(H))=d(\cP^{j^2,-1},\cP^{j^2,1}).\]
Denote by \(\cQ^{a,b}\) the slicing of the stability
condition on \(E^N\) used to construct
\(\sigma^{a,b}\).  Note that pullback and pushforward of slicings do not
increase the slicing distance. Therefore,
\[
d(\cP^{j^2,-1},\cP^{j^2,1})\le d(\cQ^{1,-1/j^2},\cQ^{1,1/j^2}),
\]
where the right-hand side converges to $0$. This implies that $\Gd\Db(X)=n$.
\end{proof}
\end{proposition}

\begin{lemma}
\label{lem:gls}
Let \(\D\) be a triangulated category with Serre functor \(\Sfun\), and let \(\sigma\) be a Bridgeland stability condition on \(\D\). Then
\[
  \gldim(\sigma) \leq n
  \quad\Longleftrightarrow\quad
  \Sfun(\sigma)\lessapprox\sigma[n].
\]
Here \(n\) may be any real number, and the right-hand side is understood via the right action of \(\R\subset\C\) on the stability space.
\end{lemma}

\begin{proof}
Assume first \(\gldim(\sigma) \leq n\).  If \(E\in\cP_\sigma(\phi)\), then 
phase inequality in \Cref{lem:kot} gives
\[
  \phi^+_\sigma(\Sfun E)
  \le
  \phi+n,
\]
Thus $\Sfun(\sigma)\lessapprox\sigma[n]$.

Conversely, assume $\Sfun(\sigma)\lessapprox\sigma[n]$.
Let \(E_i\in\cP_\sigma(\phi_i)\) and suppose $\Hom(E_1,E_2)\ne0$. By Serre duality,
\[
  \Hom(E_2,\Sfun(E_1))\ne0.
\]
Since \(\Sfun(\sigma)\lessapprox\sigma[n]\), all HN factors of \(\Sfun(E_1)\)
have phase at most \(\phi_1 + n\). Then, $\phi_2\le \phi_1+n$.
Hence $\gldim(\sigma)\le n$.
\end{proof}

\begin{corollary}
    \label{cor:key}
    Let \(X\) be a smooth projective variety of dimension \(n\), and let \(\sigma\) be a Bridgeland stability condition on \(\Db(X)\). Then
    \[
  \gldim(\sigma) = n
  \quad\Longleftrightarrow\quad
  \sigma\otimes K_X\lessapprox\sigma.
\]
\end{corollary}

\begin{proof}
    By \Cref{thm:kot-lower-bound}, we have \(\gldim(\sigma)\ge n\). The claim now follows from \Cref{lem:gls}.
\end{proof}

\begin{lemma}[{\cite[Theorem 5.16]{KOT}}]
\label{lem:genus>1}
Let \(C\) be a smooth projective curve of genus \(g(C)\ge2\). For any numerical stability condition $\sigma\in\Stab(C)$, we have 
\[
    \gldim(\sigma)>1.
  \]
\end{lemma}

\begin{lemma}
\label{lem:descent}
Let \(i:D\hookrightarrow Y\) be a smooth divisor with \(D\in |L|\).  Suppose
\[
  \sigma\otimes L\lessapprox\sigma.
\]
Then there exists a stability condition \(\sigma_D\) on \(\Db(D)\) such that
\[
  \sigma_D \otimes L|_D \lessapprox \sigma_D.
\]
\end{lemma}

\begin{proof}
Since \(\sigma\lessapprox\sigma[1]\), the assumption implies $\sigma\otimes L\lessapprox\sigma[1]$. By \cite[Proposition 1.13]{Li2025}, $i^{\sharp}\sigma$ defines a stability condition on $\Db(D)$, which we denote by \(\sigma_D\). The lemma then follows from the pullback construction.

\end{proof}

\begin{lemma}
\label{lem:big-nef-divisor}
Let \(X\) be smooth projective of dimension \(n\ge2\) over $\C$, with \(L\) big and semiample. For all sufficiently divisible \(m>0\), a general divisor
\(D\in\lvert mL\rvert\) is smooth and connected, and \(L|_D\) is also big and semiample.
\end{lemma}

\begin{proof}
Since \(L\) is semiample, $|mL|$ is base-point free for $m\gg0$. By Bertini's smoothness theorem
\cite[Chapter~III, Corollary~10.9]{Hartshorne}, a general $D\in |mL|$ is smooth. 
By Serre duality and \cite[Theorem~2.70]{KollarMori}, we have
\[ H^1\bigl(X,\OO_X(-mL)\bigr)^\vee
  \simeq
  H^{n-1}\bigl(X,K_X\otimes mL\bigr)=0.
\]
Thus, the short exact sequence
\[
  0\longrightarrow\OO_X(-mL)
  \longrightarrow\OO_X
  \longrightarrow\OO_D
  \longrightarrow0
\]
and \(H^0(\OO_X(-mL))=0\) imply
\[
  H^0(D,\OO_D)\simeq H^0(X,\OO_X)\simeq\C.
\]
It follows that \(D\) is connected.
The restriction \(L|_D\) is semiample because the surjective evaluation map
\[H^0(X,mL)\otimes\OO_X\longrightarrow\OO_X(mL)\]
restricts to a surjection
\[H^0(X,mL)\otimes\OO_D\longrightarrow\OO_D(mL|_D).\]
Moreover,
\[ L_D^{n-1} 
  =  L^{n - 1} \cdot D 
  =  mL^n > 0.
\]
Hence \(L|_D\) is big and semiample.
\end{proof}

\begin{lemma}
    Let \(X\) be a smooth projective variety and let \(L\) be a big and semiample line bundle. Then there is no stability condition satisfying
    \[\sigma \otimes L \lessapprox \sigma.\]
\end{lemma}

\begin{proof}
    We argue by induction on the dimension \(n\). If \(n=1\), there are two cases, according to whether the genus is \(g=0\) or \(g\geq1\).
    
    If \(g\geq1\), suppose for a contradiction that there exists \(\sigma\in\Stab(X)\) such that \(\sigma\otimes L\lessapprox\sigma\). Up to the action of \(\widetilde{\mathrm{GL}}_2^+(\R)\), every \(\sigma\in\Stab(X)\) is equivalent to a slope stability condition \cite{Macri}; in particular, every line bundle is stable. Iterating the assumed phase order gives \(\sigma\otimes L^m\lessapprox\sigma\) for every $m\geq1$. Since \(L\) is big and semiample, \(H^0(L^m)\neq0\) for all sufficiently large \(m\). Therefore there exists a nonzero homomorphism \(L'\to L'\otimes L^m\). Since both line bundles are stable, we have \(\phi(L')<\phi(L'\otimes L^m)\), contradicting \(\sigma\otimes L^m\lessapprox\sigma\).
    
    If \(g=0\), then \(X=\bP\) and \(L=\OO(a)\) with \(a>0\). By the classification of stability conditions on \(\Db(\bP)\), for every stability condition \(\sigma\), there exists \(k\in\Z\) such that \(\OO(k)\) and \(\OO(k+1)\) are \(\sigma\)-stable. If \(\sigma\otimes\OO(a)\lessapprox\sigma\), we have
    \begin{align*}
        \phi(\OO(k)) &< \phi(\OO(k+1)) \\
        &\leq \phi^+(\OO(k + a)) \\
        &\leq \phi(\OO(k)),
    \end{align*}
    which gives the contradiction.
    
    Therefore the lemma holds in dimension \(1\). For \(n\geq2\), assume the result in dimensions at most \(n-1\). Suppose for a contradiction that there exists $\sigma\in\Stab(X)$ with $\sigma\otimes L\lessapprox\sigma$. By iteration, $\sigma\otimes mL\lessapprox\sigma$. By \Cref{lem:big-nef-divisor}, there exists a smooth connected divisor \(Y\in|mL|\) such that \(L|_Y\) is also big and semiample. Applying \Cref{lem:descent} with $L'=mL$, we obtain a stability condition \(\sigma_Y\) on \(Y\) satisfying \(\sigma_Y\otimes L'|_Y\lessapprox\sigma_Y\), contradicting the induction hypothesis.
\end{proof}

We are now ready to complete the proof of the main result in this section.

\begin{proof}[Proof of \Cref{thm:variety}]
First, the equality $\Gd\Db(X)=n$ for any smooth projective variety $X$ follows from \Cref{pro:Gd=n}. If $-K_X$ is ample, then $\Gd\Db(X)$ is reachable by \Cref{thm:main} and \Cref{cor:key}.

Assume $K_X\in\Pic^0(X)$. By \cite[Corollary~3.5.2]{Pol07}, a connected algebraic group acts trivially on a stability condition whenever it acts trivially on the
numerical Grothendieck group. In particular, it gives $\sigma\otimes L=\sigma$ for every $L\in\Pic^0(X)$. Thus $\sigma\otimes K_X=\sigma$. Then, by \Cref{cor:key}, $\gldim$ is constantly equal to $n$ on $\Stab(\Db(X))$.

Finally, suppose that $K_X$ is big and nef. By \cite[Theorem~3.3]{KollarMori}, \(K_X\) is semiample. The preceding lemma shows that there is no stability condition such that \(\sigma\otimes K_X\lessapprox\sigma\). Therefore, by \Cref{cor:key}, we have \(\gldim(\sigma)>n\) for every stability condition $\sigma$.
\end{proof}

\begin{remark}
Let \(X\) be a smooth projective Fano variety. Applying the inducing strategy in \cite{MaWuZh:2026} yields the existence of stability conditions on the derived category of coherent sheaves on the total space of the canonical bundle \(K_X\), supported on zero section. Independently, this existence result has also been obtained recently by \cite{li2026stabilityconditionsmodulispaces} via a different method. 
\end{remark}

\subsection{From stability conditions to pre-stability conditions}
In this section, we consider the reachability problem for pre-stability conditions on smooth projective varieties. Let $\PreStab(\D)$ denote the set of pre-stability conditions on $\D$ and set
\[
\Gd^{\pre}\D\coloneq\inf_{\sigma\in\PreStab(\D)}\gldim(\sigma).
\]
Throughout this section, stability conditions are considered with respect to the full-rank lattice given by the (numerical) Grothendieck group. It is clear that $\Stab(\D)\subseteq\PreStab(\D)$ with equality when the support property is automatic. The following result shows
that this equality is exceptional.  Consequently, extending the
reachability problem to pre-stability conditions is a genuine enlargement
of the reachability problem.

\begin{proposition}
\label{prop:support}
The following statements hold.
\begin{enumerate}
  \item Let \(Q\) be a finite connected acyclic quiver and
  \(\D_Q\coloneq\Db(\modcat Q)\). Then we have
  \[
    \PreStab(\D_Q)=\Stab(\D_Q)
    \quad\Longleftrightarrow\quad
    Q\text{ is Dynkin or Euclidean}.
  \]
  \item Let \(X\) be a connected smooth projective complex variety of
  positive dimension. Then we have
  \[
    \PreStab(\Db(X))=\Stab(\Db(X))
    \quad\Longleftrightarrow\quad X\cong\bP.
  \]
\end{enumerate}
\end{proposition}

\begin{proof}
We first consider a triangulated category $\D$ with $\rk K_0(\D)<\infty$. Suppose a stability condition
\(\sigma=(Z,\cP)\) on $\D$ has semistable objects \(E_j\in\cP(\phi_j)\) with $\phi_j\to\phi_0$ and \(\cP(\phi_0)=0\). Let
\(\A=\cP(\phi_0,\phi_0+1]\) and define
\[
  W=-\Im(e^{-\pi\ii\phi_0}Z).
\]
By definition, we have $W(E)\in\R_{<0}$ for every $0\ne E\in\A$. Consider a new slicing $\cQ$ determined by
\[
\cQ(1+n)=\A[n],\text{ for } n\in\Z,
\]
and 
\[
\cQ(\phi)=0\text{ if }\phi\notin\Z.
\]
Then $(\cQ,W)$ is a pre-stability condition on $\D$. Fix any norm $\|\cdot\|$ on the finite-rank lattice. Then
\[
  \frac{|W(E_j)|}{\|E_j\|}=\frac{|Z(E_j)||\sin\pi(\phi_j-\phi_0)|}{\|E_j\|}
  \leq\sup_{\|E\|=1}|Z(E)|
  \bigl|\sin\pi(\phi_j-\phi_0)\bigr|\longrightarrow0,
\]
which implies $(\cQ,W)\notin\Stab(\D)$.

For a wild acyclic quiver, such a sequence exists by
\cite[Propositions~3.32 and~3.34]{DHKK}, so the support property does not hold automatically. For
a Dynkin quiver, since there are only finitely many indecomposable objects up to
sign, support property follows immediately. For a Euclidean quiver, let \(\delta\) be
the minimal imaginary root.  By \cite[Theorem~4.18]{Sun}, the heart is either of finite length, or
contains an object of class \(\pm\delta\).  The first case has finitely many
simple classes.  In the second, \(Z(\delta)\neq0\), and every indecomposable
object is, up to sign,  \(n\delta\) or \(\alpha+n\delta\) with \(\alpha\)
in a finite set.  Hence,
\[
  \frac{|Z(\alpha+n\delta)|}{\|\alpha+n\delta\|}
  \longrightarrow\frac{|Z(\delta)|}{\|\delta\|}>0,
\]
which proves (1).

Let $C$ be a curve of positive genus. By \cite{Macri}, for every coprime pair $(r,d),r>0$, there is a stable vector bundle of rank $r$ and degree $d$. Then the observation above produces a pre-stability condition without support property.  If
\(\dim X\geq2\), choose a positive-genus smooth complete-intersection curve
\(j:C\hookrightarrow X\) and a very ample embedding $i:X\hookrightarrow\PP^N$. As in the proof of \Cref{thm:variety}, there is a stability condition $\sigma\in\Stab(\PP^N)$ whose restrictions $\sigma_X=(\cP_X,Z_X)=i^{\sharp}\sigma\in\Stab(\Db(X))$ and $\sigma_C=j^{\sharp}\sigma_X\in\Stab(\Db(C))$ are stability conditions. Since restriction preserves semistable objects and their phases,
\[
\{\phi\in\R\colon\cP_X(\phi)\ne\{0\}\}
\]
has a dense, countable subset in $\R$. Similarly, the above observation gives a pre-stability condition without support property. Finally, due to
\(\Db(\PP^1)\simeq\Db(K_2)\), support property on $\bP$ follows from the Euclidean case of (1).
\end{proof}

The following example shows that there are triangulated categories with pre-stability conditions but no stability conditions; that is, $\PreStab(\D)\ne\emptyset$ while $\Stab(\D)=\emptyset$.

\begin{example}
\label{ex:prestab-without-stab}
Let
\[
  X=\PP^1\times\PP^1,
  \qquad
  \D=\Db(\Coh^{(1)}X),
\]
where $\Coh^{(1)}X$ is obtained by quotienting coherent sheaves by zero-dimensional sheaves on $X$. Then
\[
  K_{\rm num}(\D)\cong\Z\oplus\Pic(X)\cong\Z^3,
  \qquad
  \PreStab(\D)\neq\varnothing,
  \qquad
  \Stab(\D)=\varnothing.
\]
Indeed, the first isomorphism follows from \cite[Proposition~3.16]{MeinhardtPartsch}. Write
\[
  \Pic(X)=\Z H_1\oplus\Z H_2.
\]
Thus the class of an object of rank \(r\) and first Chern class
\(pH_1+qH_2\) is identified with \((r,p,q)\in\Z^3\).
On
\[
  K_{\rm num}(\D)_{\R}\cong\R^3
\]
we choose the Euclidean norm
\[
  \|(r,p,q)\|
  :=
  \sqrt{r^2+p^2+q^2}.
\] By
\cite[Theorem~4.6 and the subsequent remark]{MeinhardtPartsch}, every
locally finite pre-stability condition is, up to the
\(\widetilde{\mathrm{GL}}^+(2,\R)\)-action, given by the heart
\(\Coh^{(1)}(X)\) and
\[
  Z_{a,b}(E)=-(ap+bq)+\ii\rk(E),
  \qquad c_1(E)=pH_1+qH_2,\qquad a,b>0.
\]
We claim that none of them satisfies the support property on $K_0(\D)$.
For every $a,b>0$, there are integral pairs $(p_j,q_j)$ such that
\[\|(p_j,q_j)\|\to\infty,\quad|ap_j+bq_j|\leq1.\]
If \(a/b\in\QQ\), write $a/b=m/n$
with \(m,n>0\) coprime. We may then take $(p_j,q_j)=j(n,-m)$. If \(a/b\notin\QQ\), Diophantine approximation
gives infinitely many integers \(m_j,n_j\), with \(n_j\to\infty\), such
that
\[
  \left|
    \frac ab-\frac{m_j}{n_j}
  \right|
  <
  \frac1{n_j^2}.
\]
Taking $p_j=n_j$ and $q_j=-m_j$, we obtain
\[
  |ap_j+bq_j|
  =
  b\left|
    \frac ab n_j-m_j
  \right|
  <
  \frac b{n_j}
  \longrightarrow0.
\]
Now consider the line bundles \(L_j=\OO_X(p_j,q_j)\), which are \(Z_{a,b}\)-stable. Indeed, any nonzero proper subobject
of \(L_j\) in \(\Coh^{(1)}(X)\) is isomorphic to
\(L_j(-D)\) for some nonzero effective divisor \(D_j=u_jH_1+v_jH_2\), with
\(u_j,v_j\ge0\). Hence
\[
  Z_{a,b}(L_j(-D))=Z_{a,b}(L_j)+(au_j+bv_j),
\]
where \(au_j+bv_j>0\). Thus $\phi(L_j(-D_j))<\phi(L_j)$. Note that
\[
  |Z_{a,b}(L_j)|\leq\sqrt2,
  \qquad \|[L_j]\|=\|(1,p_j,q_j)\|\longrightarrow\infty.
\]
Thus $Z_{a,b}$ fails the support property on \(K_{\rm num}(\D)\). Since the support property is invariant under the
\(\widetilde{\mathrm{GL}}^+(2,\R)\)-action and implies local finiteness, we obtain
\[
  \PreStab(\D)\neq\varnothing,
  \qquad
  \Stab(\D)=\varnothing.
\]
\end{example}

The following result shows that the arguments used in \Cref{thm:variety} are largely independent of the
support property. 

\begin{proposition}
\label{prop:pre-reachability}
Let \(X\) be a smooth connected projective variety of dimension \(n\).
We have
\[
  \Gd^{\pre}\Db(X)=n.
\]
Moreover:
\begin{enumerate}
  \item if \(-K_X\) is ample, the value \(n\) is reached by a genuine
  stability condition;
  \item if \(K_X\simeq\OO_X\), every pre-stability condition has global
  dimension \(n\);
  \item if \(K_X\in\Pic^0(X)\), the value \(n\) is reachable on the pre-stability space;
  \item if \(K_X\) is big and nef, there is no pre-stability condition with
  global dimension \(n\).
\end{enumerate}
\end{proposition}

\begin{proof}
The proofs of \Cref{thm:kot-lower-bound} and \Cref{lem:kot} use only finite HN
filtrations and Serre duality. Hence they give
\(\gldim(\sigma)\geq n\) for every pre-stability condition $\sigma$ on \(\Db(X)\).
The sequence of stability conditions constructed in \Cref{pro:Gd=n} has global dimension
converging to \(n\), which proves \(\Gd^{\pre}\Db(X)=n\).  The statements (1) and (2)
follow from \Cref{thm:main,cor:key}. Assertion (3) follows from the stability conditions used
in the proof of \Cref{thm:variety}.

For (4), suppose, for a contradiction, that there is a pre-stability condition $\sigma_n\in\PreStab(\Db(X))$ with $\gldim(\sigma_n)=n$. Note that the proofs of \Cref{cor:key,lem:descent} remain valid for pre-stability conditions. Thus the argument in \Cref{thm:variety} finally produces a numerical pre-stability condition on a curve of genus $g\ge2$. It remains only to show that a curve $C$ of genus \(g\geq2\) does not admit a numerical
pre-stability condition of global dimension one.   Suppose such a pre-stability condition
\(\sigma=(Z,\cP)\) exists. By \cite[Proposition 3.5]{Qiu}, every indecomposable object is $\sigma$-semistable. After applying the $\C$-action, we may assume that $Z(r,d)=ar-d$ for some $a\in\C$. Let $E$ be any stable vector bundle. Since 
\[
\Hom(E,\OO_x)\ne0,\quad\Hom(\OO_x,E[1])\ne0,
\]
and $\phi(\OO_x)=1$ for any $x\in C$, we have $\phi(E)\in[0,1]$, which implies that $\Im a\ge0$. If $\Im a>0$, then $\phi(K_C)>\phi(\OO_C)$. However, we have $\Hom(\OO_C,K_C[1])\ne0$. Thus
\[
\gldim(\sigma)\ge1+\phi(K_C)-\phi(\OO_C)>1,
\]
which contradicts the assumption. Therefore $a\in \R$. If $a\in\QQ$, then the central charge of a stable vector bundle of slope $a$ is zero, which is impossible. Suppose now $a\in\R\backslash\QQ$. Since rational slopes of stable bundles are dense in $\R$, we can choose a stable vector bundle $E$ with 
\[
a-\deg K_C<\mu(E)<a. 
\]
Then $Z(E)\in\R_{>0}$ and hence $\phi(E)=0$. In this case, $\phi(E\otimes K_C)=1$ and Serre duality gives 
\[
\Hom(E,E\otimes K_C[1])\ne0.
\]
Thus, $\gldim(\sigma)\ge2$ which is again a contradiction. 
\end{proof}

\begin{remark}
For genuine numerical stability conditions, \(K_X\in\Pic^0(X)\) forces
\(\sigma\otimes K_X=\sigma\).  Indeed, tensoring by any
\(L\in\Pic^0(X)\) acts trivially on the numerical Grothendieck group, so
\(\sigma\) and \(\sigma\otimes L\) have the same central charge.  By
\Cref{thm:deformation}, the fiber of the forgetful map is discrete.
Since \(\Pic^0(X)\) is connected, the orbit
\(L\mapsto\sigma\otimes L\) is constant. For pre-stability conditions, this discreteness
argument does not work in general.  Thus \(K_X\in\Pic^0(X)\) gives reachability, but not constancy in
our statement.
\end{remark}

\section{g-Compactifications of stability spaces for curves}
\label{sec:compactify}
In this section, we study compactifications of stability
spaces of curves that are compatible with the global dimension function. Recall that the global dimension function is continuous, invariant under exact autoequivalences and
the \(\C\)-action on stability conditions, which implies that the global dimension function descends to a
continuous map
\[
  \gldim:
  \cM(\D)\longrightarrow
  [0,\infty].
\] 
It is therefore natural to study the compactification problem on the quotient space
\[
  \cM(\D)
  :=
  \Aut(\D)\backslash\Stab_\Lambda(\D)/\C,\]
for a triangulated category $\D$. 
For simplicity, we denote by $\cM_C$ the quotient space $\cM(\Db(C))$ associated with the curve $C$.

\begin{definition}\label{def:g-compactification}
A \emph{\(g\)-compactification} of \(\cM(\D)\) is a compact Hausdorff
space \(\overline{\cM}(\D)^{\,g}\), together with an open dense embedding
\[
  \jmath:
  \cM(\D)\hookrightarrow\overline{\cM}(\D)^{\,g},
\]
such that the global dimension function admits a continuous extension
\[
  \overline{\gldim}:
  \overline{\cM}(\D)^{\,g}
  \longrightarrow
  [0,\infty]
\]
satisfying
\[
  \overline{\gldim}\circ\jmath=\gldim.
\]
\end{definition}

Whenever possible, we require the boundary points to carry explicit limiting
stability data. For curves with positive genus, the boundary will be described by
limiting central charges and hearts. For the projective line \(\PP^1\), every boundary point will be represented by a quasi-convergent path of stability conditions. This will be made precise in the following sections.

\subsection{Positive-genus curves}
Let \(C\) be a smooth projective curve of genus \(g\ge2\). By \cite{Macri}, we have
\[
  \Stab(\Db(C))\simeq\C\times\HH,
\]
where $\HH$ denotes the complex upper half plane. Up to \(\C\)-action, a stability condition can be represented by
\(\tau=\beta+\ii\alpha\in\HH\), with heart \(\Coh(C)\) and central charge
\[
Z_{\alpha,\beta}(r,d)=-d+(\beta+\ii\alpha)r, \qquad \text{ for } \alpha>0.
\]

\begin{proposition}\label{pro:positive-formula}
Let $C$ be a smooth projective curve of genus \(g=g(C)\ge1\). Then
\[
  \gldim(\sigma_{\alpha,\beta})
  =
  1+\frac2\pi
  \arctan\!\left(\frac{g-1}{\alpha}\right).
\]
In particular, if \(g=1\), then \(\gldim\) is identically \(1\), while if
\(g\ge2\), its image is \((1,2)\).
\end{proposition}

\begin{proof}
Set \(\ell=\deg K_C=2g-2\). Tensoring by \(K_C\) preserves slope semistability, sends a bundle
of slope \(\mu\) to one of slope \(\mu+\ell\), and preserves torsion
sheaves. Serre duality gives the standard phase bound
\[
  \gldim(\sigma)
  =
  1+
  \sup_{E\ \sigma\text{-semistable}}
  \bigl(
    \phi_\sigma(E\otimes K_C)-\phi_\sigma(E)
  \bigr).
\]
For a stable bundle of slope \(\mu\), the relevant phase difference is
\[
  \frac1\pi\left(
    \arg(\beta-\mu-\ell+\ii\alpha)
    -\arg(\beta-\mu+\ii\alpha)
  \right).
\]
If $g=1$, then $\ell=0$ and the phase difference is zero. Suppose now that $g\geq2$. Rational slopes of stable bundles are
dense in $\R$, and the difference is maximal at $\beta-\mu=\ell/2$. Its maximum is
\[
  \pi-2\arctan\!\left(\frac{2\alpha}{\ell}\right)
  =2\arctan\!\left(\frac{\ell}{2\alpha}\right)=2\arctan\!\left(\frac{g-1}{\alpha}\right),
\]
which implies the required formula.
\end{proof}

Kikuta--Koseki--Ouchi  \cite{KKO} construct a Thurston compactification of the set of geometric stability conditions
\(\Geo(C)/\C\) by projectivized mass functions. For positive
genus curves, they identify the compactification with
\[
  \overline{\HH}=\HH\cup\R\cup\{\infty\}\simeq\overline\Delta,
\]
where $\overline{\Delta}$ is the closure of the unit disk 
\[
\Delta=\{z\in\C\colon|z|<1\}.
\]
However, this compactification is compatible with $\gldim$ only in genus $g=1$.
\begin{proposition}\label{prop:framed-failure}
Let $C$ be a smooth projective curve of genus $g\ge1$. Then \(\gldim\) extends continuously to \(\overline{\HH}\) if and only if
\(g=1\).
\end{proposition}
\begin{proof}
For $g=1$, $\gldim$ is constant, so this compactification is automatically
compatible. Now assume \(g\ge2\). The two sequences \(\tau_j=j+\ii\) and \(\tau_j'=\ii j\) converge to
\(\infty\in\overline{\HH}\).  By \Cref{pro:positive-formula},
\[
  \gldim(\sigma_{1,j})
  =1+\frac2\pi\arctan(g-1)>1
\]
is constant, and
\[
  \gldim(\sigma_{j,0})
  =1+\frac2\pi\arctan\!\left(\frac{g-1}{j}\right)
  \longrightarrow1.
\]
Thus $\gldim$ is not continuous at \(\infty\) when $g\ge2$.
\end{proof}
We now pass to the reduced quotient $\cM_C$. For an elliptic curve \(C\), this quotient is
\[\cM_C\simeq\PSL_2(\Z)\backslash\HH.
\]
Its standard compactification is the compactified modular curve
 \[X(1)\coloneq\PSL_2(\Z)\backslash(\HH\cup\bP(\QQ)),\]
whose boundary consists of one cusp. Since \(\gldim\) is identically one, this gives a \(g\)-compactification of $\cM_C$ by extending $\gldim$ constantly.

The cusp can be represented by the rank degeneration.  Along the path
\(\alpha\to+\infty\), after normalization, the central charges converge to
\[
  Z_{\rk}(r,d)=\ii r
\]
with the heart $\Coh(C)$.  Denote by $\Coh_0(C)\subset\Coh(C)$ the category of torsion sheaves on $C$ and $\mathrm{VB}(C)$ the category of vector bundles on $C$. The associated slicing $\cP_{0}$ is determined by
\[
\cP_0(1)=\Coh_0(C),\quad \cP_0(\frac12)=\mathrm{VB}(C).
\]
This is not an honest Bridgeland stability condition, since it vanishes on
torsion classes.

Assume now $g\ge2$. Tensoring by a line bundle of degree one sends
\(\beta\) to \(\beta+1\).   Therefore the reduced quotient
is given by
\[
  \cM_C
  :=
  \Aut(\Db(C))\backslash\Stab(\Db(C))/\C
  \simeq
  \HH/\Z
  \xrightarrow[\simeq]{\ q=e^{\ii2\pi\tau}\ }
  \Delta^\ast,
\]
where $\Delta^{\ast}\coloneq\Delta\backslash\{0\}$. Consider the reduced disk compactification
\[
  \overline\cM_C:=\overline\Delta,
\]
which is obtained by compactifying the quotient \(\HH/\Z\). By \Cref{pro:positive-formula}, $\gldim$ extends uniquely and continuously to
\[
  \overline{\gldim}:
  \overline\cM_C
  \longrightarrow[1,2],
\]
with
\[
  \overline{\gldim}(0)=1,
  \qquad
  \overline{\gldim}(q)=2\quad(|q|=1).
\]
For \(0<|q|<1\),
\[
\overline{\gldim}(q)
  =
  1+\frac2\pi
  \arctan\left(
    \frac{2\pi(g-1)}{-\log|q|}
  \right).
\]
Therefore, $\overline\cM_C$ is a \(g\)-compactification of $\cM_C$. It remains to describe the limiting stability data on the boundary. At \(q=e^{2\pi\ii\beta}\in\partial\Delta\) with $\beta\in\R/\Z$, the limiting charge is given by
  \[
    Z_\beta(r,d)=-d+\beta r.
  \]
Consider the torsion pair 
\begin{align*}
  \cF_\beta
  &=\{F\in\Coh(C)\colon F\text{ is torsion-free and }\mu_{\max}(F)<\beta\},\\
  \cT_\beta
  &=\{T\in\Coh(C)\colon T\text{ is torsion, or its torsion-free part satisfies }\mu_{\min}(T/T_{\mathrm{tor}})\geq\beta\},
\end{align*}
on $\Coh(C)$. The tilted heart is given by 
\[
\A_{\beta}\coloneq\langle\cF_{\beta}[1],\cT_{\beta}\rangle.
\]
Note that tensoring by a line bundle of degree $n$ identifies $(\A_{\beta},Z_{\beta})$ and $(\A_{\beta+n},Z_{\beta+n})$. Thus, the construction is well defined on $\overline{\cM}_{C}$. The pair \((\A_\beta,Z_\beta)\) gives the limiting boundary stability
data.  If \(\beta\notin\QQ\), this is a non-locally-finite stability
condition in the sense of \cite[Proposition~3.7]{Liu}.  If
\(\beta\in\QQ\), then some nonzero numerical classes have zero
\(Z_\beta\)-charge, and the same data is interpreted as a CLSY weak
stability condition (see \cite[Section~4.3]{Liu}). Finally, the center \(q=0\) is the same rank degeneration as the cusp in the elliptic case.

In conclusion, we obtain the following result:

\begin{proposition}
\label{pro:positive-genus curve}
Let \(C\) be a smooth projective curve with genus $g\ge1$.
\begin{enumerate}
\item If \(g(C)=1\), the standard compactified modular curve $X(1)$ is a
  \(g\)-compactification with \(\overline{\gldim}\equiv1\). 
\item If \(g(C)\ge2\), the closed disk \(\overline\cM_C\simeq\overline{\Delta^\ast}\) is a
  \(g\)-compactification on which $\overline{\gldim}$ takes the value \(1\) at the cusp and \(2\) on the
  outer circle.
\end{enumerate}
\end{proposition}

\subsection{The Projective line}
We first recall the description of $\Stab(\bP)/\C$ from \cite{Okada} and \cite{HLR}. For $k\in\Z$, let $X_k\subset\Stab(\bP)/\C$ be the locus on which  $A_k\coloneq\OO(k-1)[1]$ and $B_k\coloneq \OO(k)$ are both stable. For $\sigma\in X_k$, define 
\[
w_k(\sigma)\coloneq\log Z_\sigma(\OO(k))-\log Z_{\sigma}(\OO(k-1)). 
\]
According to \cite[Section 4.1]{HLJR}, $w_k\colon X_k\xrightarrow{\sim}\mathbb H$ is a biholomorphism. Using $\mathbb C$-action, we normalize $\sigma$ so that 
\[
m_{\sigma}(A_k)=1,\quad \phi(A_k)=1.
\]
Then, $w_k(\sigma)=\log m_\sigma(B_k)+\ii\pi d$, where 
$d\coloneq\phi_{\sigma}(B_k)>0$. 

\begin{theorem}[\cite{Okada}]\label{thm:okada-classification}
For every $\sigma\in\Stab(\bP)/\C$, there exists $\Phi\in\Aut(D^b(\bP))$ such that $\Phi\cdot\sigma\in X_0$.
Normalize $\phi(A_k)=1$ and set $d=\phi(B_k)>0$. Then
\begin{enumerate}
\item if $d<1$, all line bundles and torsion sheaves are stable;
\item if $d=1$, every indecomposable object is semistable;
\item if $d>1$, only shifts of $A_k$ and $B_k$ are semistable.
\end{enumerate}
Moreover, $\Stab(\bP)\simeq\C^2$.
\end{theorem}
In \cite[Proposition 26]{halpernleistner2024noncommutativeminimalmodelprogram}, there is an explicit biholomophism $\mathcal B\colon\C\to\Stab(\bP)/\C$ given by solutions to certain quantum differential equation such that tensoring by $\OO(1)$ acts by 
\[
\tau\mapsto\tau+\ii\pi.
\]
Hence, the quotient space
\[
\cM_{\bP}\simeq\C/(\ii\pi\Z)\xrightarrow[q=e^{2\tau}]{\simeq}\mathbb C^{\ast}.
\]

\begin{proposition}\label{thm:p1-formula}
Let $\sigma\in X_0$ and normalize it as above. Then 
\[
  \gldim(\sigma)=\max\{1,d\}.
\]
In particular, the image of global dimension on \(\Stab(\bP)\) is
\([1,\infty)\).
\end{proposition}

\begin{proof}
If \(0<d\le1\), by \Cref{thm:okada-classification}, every indecomposable
object is semistable.  Thus \(\sigma\) is totally semistable, and
\cite[Proposition~3.5]{Qiu} gives \(\gldim(\sigma)\le1\).  A skyscraper
sheaf is semistable and has a nonzero self-extension, so
\(\gldim(\sigma)\ge1\).

Assume \(d>1\). The only semistable indecomposable objects, up to shifts, are
\(A_k\) and \(B_k\) by \Cref{thm:okada-classification}. Since $A_k,B_k$ are exceptional,
their self-maps give phase difference zero. In addition, we have
\[
  \Hom(A_k[p],B_k[q])
  \cong\Ext^{q-p-1}(\OO(-1),\OO).
\]
The only nonzero degree is $q-p-1=0$ and the phase difference is $d$. In the reverse direction, 
\[
\Hom(B_k[p],A_k[q])\cong\Ext^{1+q-p}(\OO,\OO(-1))=0
\]
since both $H^0(\OO(-1))$ and $H^1(\OO(-1))$ vanish. Hence, $\gldim(\sigma)=d$ if $d>1$.
\end{proof}

The two-point compactification of \(\cM_{\bP}\) is
\(\C^*\subset\bP=\C^{\ast}\cup\{0,\infty\}\), which adds one point at each end of the cylinder. The point $0$ corresponds to $\Re\tau\to-\infty$ and $\infty$ to $\Re\tau\to+\infty$. However, the two-point compactification is not compatible with the global dimension function in the positive end $+\infty$.
\begin{proposition}\label{prop:p1-two-point-failure}
The function $\gldim$ does not extend continuously to the two-point compactification
\(\C^*\subset\bP\).
\end{proposition}
\begin{proof}
Fix $R_0\in\mathbb R_{>0}$. Consider two sequences $\{\sigma_R\colon R>R_0\},\{\sigma_R'\colon R>R_0\}\in X_0$ such that
\[
  w(\sigma_R)=R+\ii\pi,
  \qquad
  w(\sigma_R')=R+\ii\pi\sqrt R.
\]
By the convergence criterion in the proof of \cite[Theorem 6.16]{HLJR}, both sequences escape through the positive end. \Cref{thm:p1-formula} gives 
\[
\gldim(\sigma_{R})=1,\quad \gldim(\sigma_R')=\sqrt{R}\to+\infty.
\]
Hence, $\gldim$ can not extend continuously to the two-point compactification of $\cM_\bP$.
\end{proof}

Similarly as in the case of genus $g\ge2$, Kikuta--Koseki--Ouchi prove that the projectivized mass map embeds
\(\operatorname{Geo}(\mathbb P^1)/\C\), and that its closure is also a closed disk
\cite{KKO}. By \Cref{thm:okada-classification} and \Cref{thm:p1-formula}, $\gldim$ extends continuously, and constantly, to this
geometric closed disk. However, this does not compactify the whole stability space.  In fact the mass map on
\(\Stab(\bP)/\C\) is non-injective for every choice of test objects
\cite[Proposition~6.3]{KKO}. The missing information is precisely phase information in exceptional chambers $X_k$.

We now introduce a refined compactification of $\cM_\bP$ which is compatible with $\gldim$, based on \cite{HLJR}. Let $\sigma_t$ be a path of stability conditions on a triangulated category $\D$. Denote by $\phi^+_t(E)$ and $\phi^-_t(E)$ the largest and smallest phase of an HN factor of a nonzero object $E$ with respect to $\sigma_t$, respectively. For any $E\in\D$, define 
\[
\phi_t(E)\coloneq\frac{1}{m_t(E)}\sum_i\phi_t(F_i)\cdot|Z_t(F_i)|,
\]
and 
\[
 \ell_t(E)
  :=\log m_t(E)+i\pi\phi_t(E),
  \qquad
  \ell_t(E/F):=\ell_t(E)-\ell_t(F),
\]
where $F_i$ are the HN factors of $E$ with respect to $\sigma_t$. An object $E$ is called \emph{limit semistable} if
\[
  \phi_t^+(E)-\phi_t^-(E)\longrightarrow0.
\]
Following \cite[Definition~2.8]{HLJR}, a path $\sigma_t$ is called 
\emph{quasi-convergent} if
\begin{enumerate}
\item every object admits a limit HN filtration whose successive limit
  semistable factors have asymptotically strictly decreasing phases;
\item for every pair of limit semistable objects $E,F$, the limit
  \[
    \lim_{t\to\infty}
    \frac{\ell_t(E/F)}{1+|{\ell_t(E/F)}|}
 \]
  exists.
\end{enumerate}
A path $\sigma_t$ in $\Stab(\D)/\C$ is 
quasi-convergent if any lift of $\sigma_t$ to $\Stab(\D)$ is
quasi-convergent. For \(\Db(\bP)\), the compactification obtained from the limiting types of such paths admits a concrete description due to \cite[Section 6.3.2]{HLR},

\begin{theorem}[\cite{HLR}]
\label{thm:p1-qc-disk}
There is a compact Hausdorff space
\[
  \overline\cM_{\bP}
  \simeq
  \frac{S^1\times[-\infty,+\infty]}
       {S^1\times\{-\infty\}\ \mathrm{collapsed}}
  \simeq\overline{\Delta^*},
\]
whose interior is \(\cM_{\bP}\simeq\C^*\).  Its boundary points have the following types:
\begin{enumerate}
\item a cone point \(c_0\), represented by quasi-convergent paths satisfying 
\[
w_k(t)=\frac{\ii}{t};
\]
\item generic outer points \(b_\vartheta\), \(0<\vartheta<1\), represented
   by quasi-convergent paths satisfying
  \[
    \frac{w(t)}{1+\abs{w(t)}}
    \longrightarrow e^{\ii\pi\vartheta};
  \]
\item one distinguished outer point \(b_*\), represented by quasi-convergent paths satisfying
  \[
    \frac{w(t)}{1+\abs{w(t)}}\longrightarrow1.
  \]
\end{enumerate}
\end{theorem}
\begin{proof}
This follows from the explicit calculation of \(\bP\) in \cite[Theorem~6.16 and Lemma~6.17]{HLR}.
\end{proof}

Consider the homeomorphism
\begin{equation}\label{eq:rho}
  \rho:[1,\infty]\xrightarrow{\sim}[0,1],
  \qquad
  \rho(t)=1-\frac1t,
  \qquad
  \rho(\infty)=1,
\end{equation}
and let \(\iota:\cM_{\bP}\hookrightarrow\overline\cM_{\bP}\) be the
open embedding.  Define
\begin{equation}\label{eq:g-compactification}
  \overline\cM_{\bP}^{\,g}
  :=
  \overline{\left\{
    \bigl(\iota(\sigma),\rho(\gldim(\sigma))\bigr):
    \sigma\in\cM_{\bP}
  \right\}}
  \subset\overline\cM_{\bP}\times[0,1].
\end{equation}

\begin{proposition}
\label{prop:p1-refinement}
The following statements hold.
\begin{enumerate}
\item The function $\gldim$ extends continuously to
\(\overline\cM_{\bP}\setminus\{b_*\}\), with
\[
  \overline{\gldim}(c_0)=1,
  \qquad
  \overline{\gldim}(b_\vartheta)=\infty
  \quad(0<\vartheta<1).
\]
It has no continuous extension at \(b_\ast\). More precisely, its cluster
set at $b_\ast$ is \([1,\infty]\).
\item The space \(\overline\cM_{\bP}^{\,g}\) is compact Hausdorff,
contains \(\cM_{\bP}\) as an open dense subspace, and carries a
continuous extension
\[
  \overline{\gldim}:
  \overline\cM_{\bP}^{\,g}\longrightarrow[1,\infty].
\]
For the first projection
\(\pi:\overline\cM_{\bP}^{\,g}\to\overline\cM_{\bP}\), we have
\[
  \pi^{-1}(x)=\{x\}\,\,\,\text{ if }x\neq b_*,
  \qquad
  \pi^{-1}(b_*)\simeq[1,\infty].
  \]
  \item Every boundary point of
\(\overline\cM_{\bP}^{\,g}\) is represented by a quasi-convergent
path.  
\end{enumerate}
Consequently, \(\overline\cM_{\bP}^{\,g}\) is a g-compactification of $\cM_{\bP}$.
\end{proposition}

\begin{proof}
By \Cref{thm:p1-qc-disk}, at the cone point \(c_0\), the exceptional coordinate satisfies
$w(t)\to0$. Hence,
\[
d(t)=\frac{\Im w(t)}{\pi}\longrightarrow0.
\]
By \Cref{thm:p1-formula},
\[
  \gldim(w)=\max\{1,d\}\longrightarrow1.
\]
Now let a path converge to \(b_\vartheta\), where
\(0<\vartheta<1\).  By \Cref{thm:p1-qc-disk},
\[
  \frac{w(t)}{1+|w(t)|}
  \longrightarrow e^{\ii\pi\vartheta}.
\]
It follows that
\[
\operatorname{Im}w(t)\longrightarrow+\infty,
  \qquad
  d(t)=\frac{\operatorname{Im}w(t)}{\pi}\longrightarrow+\infty,
\]
and therefore \(\gldim(\sigma_t)\to\infty\).

It remains to analyze \(b_\ast\). For \(t\geq1\) and \(s\in[0,1]\), set
\[
  d_t(s)
  :=1+\frac1t+
    \frac{s\sqrt t}{1+(1-s)\sqrt t},
  \qquad
  w_{t,s}:=t+\ii\pi d_t(s).
\]
Then \(d_t(s)>1\) and
\begin{equation}\label{eq:universal-limits}
  \frac{w_{t,s}}{1+\abs{w_{t,s}}}\longrightarrow1,
  \qquad
  \lim_{t\to\infty}d_t(s)
  =g(s):=
  \begin{cases}
    (1-s)^{-1},&0\leq s<1;\\
    \infty,&s=1.
  \end{cases}
\end{equation}
Thus every path in this family converges to \(b_*\) in $\overline{\cM}_{\bP}$ and 
\[
  \lim_{t\to\infty}\gldim(w_{t,s})
  =
  \begin{cases}
    \dfrac1{1-s}, & 0\le s<1,\\[6pt]
    \infty, & s=1.
  \end{cases}
\]
This implies the cluster set at $b_{\ast}$ contains
\[
  \left\{
    \frac1{1-s}\colon 0\le s<1
  \right\}
  \cup\{\infty\}
  =
  [1,\infty].
\]
In particular, to prove that \(\gldim\) admits no continuous extension at \(b_\ast\), it remains to show that the paths \(w_{t,s}\) are quasi-convergent, which will be verified below.

Since the product \(\overline\cM_{\bP}\times[0,1]\) is compact Hausdorff, the
closed subset \(\overline\cM_{\bP}^{\,g}\) is also compact Hausdorff.  Over the
interior \(\cM_{\bP}\), it is the graph of the continuous function
\(\rho\circ\gldim\), hence \(\cM_{\bP}\) embeds as an open dense subspace.
For \(x\in\overline\cM_{\bP}\), the fiber of \(\pi\) is precisely the cluster
set
\[
  \pi^{-1}(x)
  =
  \{x\}\times
  \operatorname{Clust}_x(\rho\circ\gldim).
\]
The above boundary computations give that $\pi^{-1}(x)=\{x\}$ for
\(x\neq b_\ast\) and
\[
  \operatorname{Clust}_{b_\ast}(\rho\circ\gldim)
  =
  \rho([1,\infty])
  =
  [0,1].
\]
Equivalently,
\[
  \pi^{-1}(b_\ast)\simeq[1,\infty].
\]
For \(g\in[1,\infty]\), choose \(s=1-1/g\), with \(s=1\) when
\(g=\infty\).

We now verify that, for every fixed $s\in[0,1]$, the path $t\mapsto w_{t,s}$ is quasi-convergent. The representatives of \(c_0\) and \(b_\vartheta\) for $0<\vartheta<1$, are quasi-convergent by
\Cref{thm:p1-qc-disk}. By \Cref{thm:okada-classification}, along \(w_{t,s}\), the
stable factors are \(A_k\) and \(B_k\) up to shifts with
\[
  \phi_{t,s}(A_k[p])=1+p,
  \qquad
  \phi_{t,s}(B_k[q])=d_t(s)+q.
\]
Therefore
\[
  \phi_{t,s}(B_k[q])-\phi_{t,s}(A_k[p])
  =
  d_t(s)+q-p-1.
\]
For \(s<1\), this has a finite limit and for \(s=1\), it tends to \(+\infty\). By \Cref{lem:uniform}, the stable factors of any fixed object
belong to a finite collection independent of $t$, and their order is
eventually constant. Group together consecutive factors whose phase
differences tend to zero. Each resulting subquotient $G$ satisfies
\[
  \phi^+_{t,s}(G)-\phi^-_{t,s}(G)\longrightarrow0,
\]
and is therefore limit semistable. The limiting phase gap between two
successive groups is strictly positive. Hence these groups form a
limit HN filtration.

It remains to verify the logarithmic condition. Suppose first that $s<1$, and set
\[
  \lambda_s:=\lim_{t\to\infty}d_t(s)=\frac{1}{1-s}.
\]
If $E$ is limit semistable, then all stable factors in its
HN filtration have the same limiting phase.  Therefore, these factors are either
all copies of one shift $A_k[p]$, all copies of one shift $B_k[q]$, or,
when $\lambda_s=N\in\mathbb Z$, copies of both $A_k[p]$ and $B_k[q]$
satisfying
\[
  1+p=N+q.
\]
The last possibility is precisely the case in which the limiting phase
gap is an integer. When $s=1$, we have $d_t(1)\to+\infty$, so the third case cannot occur. Let $a_E$ and $b_E$ denote
the total multiplicities of its $A_k[p]$- and $B_k[q]$-factors in the fixed
finite collection of \Cref{lem:uniform}. In particular,
$a_E$ and $b_E$ are independent of $t$. Under our normalization,
\[
  m_{t,s}(A_k[p])=1,
  \qquad
  m_{t,s}(B_k[q])=e^t.
\]
Hence
\[
  m_{t,s}(E)=a_E+b_Ee^t.
\]
If $b_E=0$, then
\[
  \ell_{t,s}(E)
  =
  \log a_E+\ii\pi(1+p).
\]
If $b_E>0$, then
\[
\begin{aligned}
  \ell_{t,s}(E)
  &=
  \log(a_E+b_Ee^t)
  +
  \ii\pi
  \frac{
    a_E(1+p)+b_Ee^t(d_t(s)+q)
  }{
    a_E+b_Ee^t
  }\\
  &=
  w_{t,s}+\log b_E+\ii\pi q+o(1).
\end{aligned}
\]
Indeed, the difference between the two sides is
\[
  \log\left(1+\frac{a_E}{b_E}e^{-t}\right)
  +
  \ii\pi
  \frac{a_E}{a_E+b_Ee^t}
  \bigl(1+p-d_t(s)-q\bigr),
\]
which tends to zero. It follows that every limit-semistable object $E$ satisfies exactly one
of the following asymptotic formulas:
\[
  \ell_{t,s}(E)=c_E+o(1)
\]
or
\[
  \ell_{t,s}(E)=w_{t,s}+c_E+o(1)
\]
for some constant $c_E\in\mathbb C$. Consequently, for any two
limit-semistable objects $E$ and $F$, we have
\[
  \ell_{t,s}(E/F)=c_{E,F}+o(1)
\]
or
\[
  \ell_{t,s}(E/F)
  =
  \pm w_{t,s}+c_{E,F}+o(1)
\]
for some $c_{E,F}\in\mathbb C$. Set
\[
 u_s:=\lim_{t\to\infty}
 \frac{w_{t,s}}{1+|w_{t,s}|}=
 \begin{cases}
 1,&0\leq s<1,\\[2mm]
 \dfrac{1+i\pi}{\sqrt{1+\pi^2}},&s=1.
 \end{cases}
\]

It follows that
\[
 \lim_{t\to\infty}
 \frac{\ell_{t,s}(E/F)}{1+|\ell_{t,s}(E/F)|}
 =
 \begin{cases}
 \dfrac{c_{E,F}}{1+|c_{E,F}|},
   &\ell_{t,s}(E/F)=c_{E,F}+o(1),\\[3mm]
 u_s,
   &\ell_{t,s}(E/F)=w_{t,s}+c_{E,F}+o(1),\\[2mm]
 -u_s,
   &\ell_{t,s}(E/F)=-w_{t,s}+c_{E,F}+o(1).
 \end{cases}
\]
Thus the logarithmic condition in the definition of quasi-convergence
is satisfied. Hence the paths
\(t\mapsto w_{t,s}\) are quasi-convergent.

Finally, given \(L\in[1,\infty]\), take $s=1-\frac1L$ with the convention \(s=1\) for \(L=\infty\). By computations above, the path $w_{s,t}$ represents the point of
\(\pi^{-1}(b_\ast)\) with 
\[
  \lim_{t\to\infty}\gldim(w_{t,s})=L.
\]
Thus every boundary point of $\overline\cM_{\bP}^{\,g}$ has a quasi-convergent representative.
\end{proof}

\begin{lemma}\label{lem:uniform}
Fix $s\in[0,1]$ and an object
$E\in D^b(\mathbb P^1)$. There exist finite subsets
$I_A(E),I_B(E)\subset\mathbb Z$ and multiplicities
$a_p,b_q\in\mathbb Z_{\geq0}$, independent of $t$, such that every
stable factor occurring in the $\sigma_{t,s}$-HN filtration of $E$ is
contained in the fixed finite collection
\[
  \left\{
    A_k[p]^{\oplus a_p}:p\in I_A(E)
  \right\}
  \cup
  \left\{
    B_k[q]^{\oplus b_q}:q\in I_B(E)
  \right\}.
\]
Moreover, after increasing $t$ sufficiently, the order of these factors
is constant, up to grouping factors whose phase differences tend to
zero.
\end{lemma}

\begin{proof}
We first show that, for every fixed object
$E\in D^b(\mathbb P^1)$, all stable factors in its HN
filtration belong to a finite collection independent of $t$.
Since $\operatorname{Coh}(\mathbb P^1)$ is hereditary, there is an
isomorphism
\[
  E\simeq\bigoplus_r H^r(E)[-r],
\]
with only finitely many nonzero summands. It is therefore enough to consider line bundles and torsion sheaves on $\bP$. For $n\geq k$, the standard exact sequence
\[
  0
  \longrightarrow
  \mathcal O(k-1)^{\oplus(n-k)}
  \longrightarrow
  \mathcal O(k)^{\oplus(n-k+1)}
  \longrightarrow
  \mathcal O(n)
  \longrightarrow0
\]
gives a filtration of $\mathcal O(n)$ with factors $B_k^{\oplus(n-k+1)} $ and $A_k^{\oplus(n-k)}$. Their phases are $d_t(s)$ and $1$, respectively, so this is the HN
filtration because $d_t(s)>1$. For $n\leq k-1$, the exact sequence
\[
  0
  \longrightarrow
  \mathcal O(n)
  \longrightarrow
  \mathcal O(k-1)^{\oplus(k-n)}
  \longrightarrow
  \mathcal O(k)^{\oplus(k-n-1)}
  \longrightarrow0
\]
gives, after rotation, a filtration with factors $B_k[-1]^{\oplus(k-n-1)}$ and $
  A_k[-1]^{\oplus(k-n)}$. Their phases are $d_t(s)-1$ and $0$, respectively, so this is again the
HN filtration. Finally, if $T$ is a torsion sheaf of length $\ell$, its standard
resolution
\[
  0
  \longrightarrow
  \mathcal O(k-1)^{\oplus\ell}
  \longrightarrow
  \mathcal O(k)^{\oplus\ell}
  \longrightarrow
  T
  \longrightarrow0
\]
gives an HN filtration with factors $B_k^{\oplus\ell}$ and $A_k^{\oplus\ell}$. Applying these descriptions to the finitely many $H^r(E)[-r]$, we obtain a finite multiset
\[
  \mathcal S_E
  =
  \{A_k[p]\text{ with fixed multiplicities}\}
  \cup
  \{B_k[q]\text{ with fixed multiplicities}\},
\]
which is independent of $t$. 

The order of two $A_k$-factors or two $B_k$-factors is independent of
$t$. Thus the order can change only when
\[
  1+p=d_t(s)+q
\]
for one of the finitely many pairs $(p,q)$ in
$\mathcal S_E$. For fixed $s$, $d_t(s)$ is
eventually monotone. Indeed, after writing $x=\sqrt t$, we have
\[
  d_t(s)
  =
  1+\frac1{x^2}
  +
  \frac{sx}{1+(1-s)x}
\]
and hence
\[
  \frac{d}{dx}d_{x^2}(s)
  =
  -\frac{2}{x^3}
  +
  \frac{s}{(1+(1-s)x)^2}.
\]
This is negative for $s=0$ and positive for all sufficiently large
$x$ when $s>0$. Consequently, each of the finitely many relevant
equalities is crossed only finitely many times. It follows that, for every fixed object $E$, the HN ordering is
constant for all sufficiently large $t$, except that some consecutive
factors may have phase differences tending to zero. 
\end{proof}

\begin{remark}
Topologically, \(\overline\cM_{\bP}^{\,g}\) is obtained from the
quasi-convergent closed disk by replacing the single boundary point \(b^*\)
with the compact interval \([1,\infty]\).  This is analogous to a real
blow-up in that one point, which is replaced by the missing limiting extremal Hom-phase width.
\end{remark}
\section{Algebraic models}
\label{sec:algebraic}

In this section, we discuss algebraic aspects of the reachability problem for the global dimension function. We first consider finite-dimensional algebras, and then explain how semiorthogonal decompositions give stability conditions with large global dimension.

\subsection{Finite-dimensional algebras}
\label{subsec:finite-dimensional-algebras}

Let $A$ be a nonzero finite-dimensional algebra over $\C$, and set
\[
  \A:=\modcat\text{-}A,
  \qquad
  \D:=\Db(\A),
\]
where $\modcat\text{-}A$ denotes the category of finite-dimensional
right $A$-modules.  The homological global dimension is
\[
 \algldim A
 :=\sup\bigl\{m\ge0\,\bigm|\,
 \Ext_A^m(M,N)\ne0
 \text{ for some }M,N\in\A\bigr\}.
\]

\begin{proposition}\label{thm:algebra}
Let $A$ be a finite-dimensional algebra over $\C$.
\begin{enumerate}
 \item If $\algldim A=n<+\infty$, there exists
 $\sigma\in\Stab(\Db(\A))$ such that
 $\gldim(\sigma)=n$.
 \item If $\algldim A=+\infty$, then every
 $\sigma\in\Stab(\Db(\A))$ satisfies
 $\gldim(\sigma)=+\infty$.
\end{enumerate}
\end{proposition}

\begin{proof}
Let $S_1,\dots,S_r$ be representatives of the simple right $A$-modules.
Since $\A$ is a finite-length category, their classes form a basis of
$K_0(\A)=K_0(\Db(\A))$.  Define
\[
 Z_{\mathrm{std}}([S_i])=-1
 \qquad(1\le i\le r).
\]
Thus, for every nonzero $M\in\A$,
\[
 Z_{\mathrm{std}}(M)=-\ell(M)\in\R_{<0},
\]
where $\ell(M)$ is the composition length of $M$. Then $\sigma_{\mathrm{std}}=(\cP_{\mathrm{std}},Z_{\mathrm{std}})$ is a stability condition on $\Db(\A)$. Every nonzero semistable object is of the form $M[p]$, with $0\ne M\in\A$,
and has phase $1+p$.  For $M,N\in\A$ and $p,q\in\Z$,
\[
 \Hom_{\D}(M[p],N[q])
 \cong \Hom_{\D}(M,N[q-p])\cong\Ext_A^{q-p}(M,N).
\]
Taking the supremum over all $M,N,p,q$ therefore gives
\[
 \gldim(\sigma_{\mathrm{std}})
 =\algldim A.
\]
Now assume that $\algldim A=\infty$.  Thus some simple module, say
$S_i$, has an infinite minimal projective resolution
\[
 \cdots\longrightarrow P_m\longrightarrow P_{m-1}
 \longrightarrow\cdots\longrightarrow P_0\longrightarrow S_i
 \longrightarrow0.
\]
For each $m$, choose a simple quotient $S_{j(m)}$ of $P_m$.  Minimality of
the resolution implies that, after applying $\Hom_A(-,S_{j(m)})$, all
differentials are zero.  Therefore
\[
 \Ext_A^m(S_i,S_{j(m)})
 \cong\Hom_A(P_m,S_{j(m)})\ne0.
\]

Now fix an arbitrary $\sigma\in\Stab(\D)$. For every $m\ge0$, we have
\[
 \gldim(\sigma)
 \ge m+\phi^-_\sigma(S_{j(m)})-\phi^+_\sigma(S_i).
\]
Since there are only finitely many simple modules, the constant
\[
 C:=\phi^+_\sigma(S_i)-\min_{1\le j\le r}\phi^-_\sigma(S_j)
\]
is finite.  Hence $\gldim(\sigma)\ge m-C$ for every $m$, and letting
$m\to+\infty$ yields $\gldim(\sigma)=+\infty$.
\end{proof}  

When $\algldim A<\infty$, \Cref{thm:algebra} proves the
reachability of the value $\algldim A$, but it does not assert that this value is the categorical infimum. Indeed, $\algldim A$ depends on the chosen algebraic heart and is
not invariant under derived equivalence. By \Cref{thm:kot-lower-bound}, we know that $\Sdimup(\Db(\A))\le\Gd\Db(\A)$. In the following example, we confirm they are equal for the categories arising from
homologically smooth graded gentle algebras, which gives an affirmative answer to the question raised by \cite{CES}.

\begin{example}\label{ex:surface}
Let $(\Sigma,M,\eta)$ be a graded marked surface with no interior stops, assume
that every boundary component contains a marked point, and put
\[
  \D=\cW(\Sigma,M,\eta).
\]
Equivalently, $\D$ is the perfect derived category of the homologically smooth
graded gentle algebra associated with a full formal arc system
\cite{HKK,LP}.  Assume throughout that $\Stab(\D)\neq\varnothing$.

First suppose that $\Sigma$ is not a disk.  Write
$\partial\Sigma=\bigsqcup_{i=1}^b\partial_i\Sigma$, let $m_i$ be the number of
marked points on $\partial_i\Sigma$, and let $\omega_i$ be the winding number
in the convention of \cite{CES}.  If $w_i$ denotes the clockwise winding
number in \cite{QiuFlow}, then $w_i=-\omega_i$.  The formula of
Chang--Elagin--Schroll \cite[Theorem~2.11]{CES} gives
\[
  \Sdimup\D
  =1-\min\left\{0,\frac{\omega_1}{m_1},\ldots,
                    \frac{\omega_b}{m_b}\right\}
  =1+\max\left\{0,\frac{w_1}{m_1},\ldots,
                    \frac{w_b}{m_b}\right\}.
\]
On the other hand, Qiu's critical-value theorem
\cite[Corollary~5.18]{QiuFlow} says that
\[
  \Gd\D\in
  \cV(\Sigma,M,\eta)
  :=\left\{1+\frac{w_i}{m_i}\ \middle|\ w_i\geq0\right\}.
\]
The Poincar\'e--Hopf identity is
\[
  \sum_{i=1}^b w_i=4g-4+2b\geq0,
\]
where the last inequality uses the assumption that $\Sigma$ is not a disk.
Thus $\cV(\Sigma,M,\eta)$ is nonempty and
\[
  \max\cV(\Sigma,M,\eta)=\Sdimup\D.
\]
Consequently, \Cref{thm:kot-lower-bound} gives
\[
  \Sdimup\D
  \leq \Gd\D
  \leq \max\cV(\Sigma,M,\eta)
  =\Sdimup\D.
\]
Hence
\begin{equation*}
  \Gd\D=\Sdimup\D
  =1+\max\left\{0,\frac{w_1}{m_1},\ldots,
                    \frac{w_b}{m_b}\right\}.
\end{equation*}

Now suppose that $\Sigma$ is a disk with $m\geq2$ marked points.  Then
\[
  \D\simeq \Db(\modcat A_{m-1}),
  \qquad \Sfun^m\simeq[m-2].
\]
It follows that $\Sdimup\D=(m-2)/m=1-2/m$.  The Dynkin Gepner stability
condition $\sigma_{\mathrm G}$ satisfies
$\gldim(\sigma_{\mathrm G})=1-2/m$
\cite[Theorem~3.1]{QiuFlow}.  Therefore
\[
  1-\frac2m=\Sdimup\D
  \leq\Gd\D
  \leq\gldim(\sigma_{\mathrm G})
  =1-\frac2m.
\]
Combining the two cases, we obtain
\begin{equation*}
\Gd\D
=
\Sdimup\D
=
\begin{cases}
\displaystyle
1+\max\left\{
0,\frac{w_1}{m_1},\ldots,\frac{w_b}{m_b}
\right\},
& \Sigma\text{ is not a disk},\\[8pt]
\displaystyle
1-\frac{2}{m},
& \Sigma\text{ is a disk with }m\text{ marked points}.
\end{cases}
\end{equation*}

The nonemptiness assumption on the stability space is essential here. The standard-graded
once-punctured torus with one marked point may have empty stability space
\cite[Corollary~1.3]{HW}. 
\end{example}

\subsection{Unboundedness from semiorthogonal decompositions}
We next study the upper range of the global dimension function.
Let $\D$ be the homotopy category of a smooth, proper,
idempotent-complete pretriangulated dg-category, and let
\begin{equation}\label{eq:polarizable-sod}
  \D=\langle\D_1,\ldots,\D_s\rangle
\end{equation}
be a semiorthogonal decomposition.  We say
\eqref{eq:polarizable-sod} \emph{polarizable} if
$\Stab(\D_i)\neq\varnothing$ for every $i$.  It is
\emph{nonorthogonal} if there exist $i<j$, objects
$X\in\D_i$, $Y\in\D_j$, and $q\in\Z$ such that
\[
  \Hom_{\D}(X,Y[q])\neq0.
\]

Choose finite-rank lattices $\Lambda_i$ and stability conditions
\[
  \sigma_i=(\cP_i,Z_i)\in\Stab_{\Lambda_i}(\D_i).
\]
Set $\Lambda:=\bigoplus_{i=1}^s\Lambda_i$.
\begin{proposition}\label{prop:sod-unbounded}
Suppose that \eqref{eq:polarizable-sod} is polarizable and nonorthogonal.
Then
\[
  \overline{\Gd}\D=+\infty.
\]
\end{proposition}

\begin{proof}
For $t\geq0$ and $1\leq i\leq s$, define
\[
  z_i(t):=\pi\ii(i-1)t
\]
and
\[
  \sigma_{i,t}
  :=
  \sigma_i\cdot\bigl(-(i-1)t\bigr).
\]
Thus
\[
  Z_{\sigma_{i,t}}
  =
  e^{\pi\ii(i-1)t}Z_i,
  \qquad
  \cP_{\sigma_{i,t}}(\phi)
  =
  \cP_i\bigl(\phi-(i-1)t\bigr).
\]
For $i<j$ one has
\[
  \frac{z_j(t)-z_i(t)}
       {1+\lvert z_j(t)-z_i(t)\rvert}
  \longrightarrow \ii.
\]
By \cite[Theorem~3.15(1)]{HLJR}, there exists a stability condition $\sigma_t=(\cP_t,Z_t)\in\Stab_{\Lambda}(\D)$ satisfying 
\[
Z_t|_{K_0(\D_i)}=e^{\ii \pi(i-1)t}Z_i\]
and 
\begin{equation}
\label{eq:glued-slicing}
\cP_i(\phi)\subseteq\cP_t(\phi+(i-1)t)
\end{equation}
for every $\phi\in\R$ and $t$ sufficiently large. In particular, every $\sigma_{i,t}$-semistable object remains
$\sigma_t$-semistable. By nonorthogonality, there exist $i<j$, objects $X\in\D_i$ and $Y\in\D_j$, and $q\in\Z$ such that
\[
  \Hom_{\D}(X,Y[q])\neq0.
\]
Consider the finite HN filtrations of $X$ and $Y$ with respect to
$\sigma_i$ and $\sigma_j$, respectively. Tracing the nonzero morphism $X\to Y[q]$
through the long exact $\Hom$-sequences associated with the triangles
in these filtrations, we obtain a $\sigma_i$-semistable HN factor
$E\in\D_i$ of $X$, a $\sigma_j$-semistable HN factor
$F\in\D_j$ of $Y$, and an integer $q'\in\Z$ such that
\[
  \Hom_{\D}(E,F[q'])\neq0.
\]
Note that the integer $q'$ need not coincide with $q$, but it is independent of $t$. Write $\phi_{\sigma_i}(E)=\alpha$ and $\phi_{\sigma_j}(F)=\beta$. By \eqref{eq:glued-slicing}, both $E$ and $F[q']$ are
$\sigma_t$-semistable, with phases
\[
  \phi_{\sigma_t}(E)=\alpha+(i-1)t,
  \qquad
  \phi_{\sigma_t}(F[q'])
    =\beta+q'+(j-1)t.
\]
Consequently,
\[
  \gldim(\sigma_t)
  \geq
  \phi_{\sigma_t}(F[q'])-\phi_{\sigma_t}(E)
  =
  \beta+q'-\alpha+(j-i)t.
\]
Since $j-i>0$, the right-hand side tends to $+\infty$.
\end{proof}

\begin{corollary}
Let $X$ be a connected smooth projective variety.  Then
\[
 \overline{\Gd}\D=+\infty
\]
in each of the following cases:
\begin{enumerate}
  \item $X$ admits a full exceptional collection of length at least two;
  \item $X$ is a Fano threefold of Picard rank one;
  \item $X$ is a cubic fourfold;
  \item $X$ is a Gushel--Mukai variety of dimension $3,4,5$, or $6$;
  \item $X$ is a cubic fivefold.
\end{enumerate}
\end{corollary}

\begin{proof}
For connected $X$, an orthogonal decomposition would give a nontrivial idempotent
of the identity functor, hence a nontrivial idempotent in
$H^0(X,\OO_X)$, contradicting connectedness. By \Cref{prop:sod-unbounded}, it is enough to show the existence of nontrivial polarizable semiorthogonal decompositions for each case. The first case then follows immediately. In each of the remaining cases, there is a standard nontrivial
semiorthogonal decomposition
\[
  D^b(X)
  =
  \bigl\langle
    \mathrm{Ku}(X),E_1,\ldots,E_m
  \bigr\rangle,
\]
where $E_1,\ldots,E_m$ are exceptional objects. The existence of stability conditions on $\mathrm{Ku}(X)$ is proved for the Fano threefolds and cubic
fourfolds in \cite[Theorems~1.1 and~1.2]{BLMS}, for Gushel--Mukai varieties
in \cite[Theorem~1.2]{PPZ}, and for cubic fivefolds in
\cite[Theorem~1.1]{LiuCubicFivefold}.  Thus these decompositions are
polarizable and nonorthogonal.
\end{proof}

\begin{remark}
The nonorthogonality hypothesis is necessary.  For example,
\[
  \Db(\C)\oplus\Db(\C)
  =\langle\Db(\C),\Db(\C)\rangle
\]
is a polarizable orthogonal decomposition, but there are no cross morphisms
and hence $\gldim\equiv0$.  
\end{remark}

\begin{remark}
The function $\Gd$ does not need to be monotone under admissible
embeddings.  Let $X=\FF_3$ be the Hirzebruch surface with a fiber $F\subset X$ and $S\subset X$ the $(-3)$-curve. Elagin--Lunts
\cite[Example~5.15]{ElaginLunts} consider
\[
  E:=R_{\OO_X(S)}(\OO_X),
  \qquad
  \B:=\langle\OO_X(-F),E\rangle\subset\Db(X).
\]
The pair $(\OO_X(-F),E)$ is exceptional, and
$\Hom^k(\OO_X(-F),E)\neq0$ for $k=-1,0,1$.  Thus $\B$ is an admissible
smooth proper subcategory and
\[
  \Sdimup(\B)=3.
\]
Its exceptional decomposition is polarizable, so $\Stab(\B)\neq\varnothing$. By \Cref{thm:kot-lower-bound}, we have $\Gd(\B)\geq3$. On the other hand,
\Cref{thm:variety} yields $\Gd \Db(X)=2$.  Hence
\[
  \Gd(\B)>\Gd(\Db(\FF_3)).
\]
\end{remark}

\bibliographystyle{alpha}
\bibliography{corrected}

\end{document}